\documentclass[12pt]{amsart}
\usepackage{latexsym,amsfonts,amssymb,amsmath,amsthm,verbatim} 
\newcommand{\N}{\ensuremath{\mathbb N}}
\newcommand{\F}{\ensuremath{\mathbf F}}
\newcommand{\Z}{\ensuremath{\mathbb Z}}

\newcommand{\R}{\ensuremath{\mathbb R}}

\newcommand{\Q}{\ensuremath{\mathbb Q}}
\newcommand{\C}{\ensuremath{\mathbb C}}

\newcommand{\ri}{\ensuremath{\mathcal{O_{\mathbf{F}}}}} 

\newcommand{\ec}{\textrm}

\newcommand{\congru}{\equiv}
\newcommand{\barre}{\overline }

\newcommand\m{\mu}

\makeatletter
\newcommand{\vast}{\bBigg@{4}}
\newcommand{\Vast}{\bBigg@{5}}
\makeatother

\theoremstyle{plain}		
	\newtheorem{theorem}{Theorem}[section]
	\newtheorem{prop}[theorem]{Proposition}
	\newtheorem{cor}[theorem]{Corollary}
     \newtheorem{lemma}[theorem]{Lemma}
	\newtheorem{definition}[theorem]{Definition}
	\newtheorem{conj}[theorem]{Conjecture}

\theoremstyle{remark}		
	
	\newtheorem*{remark}{Remark}
	\newtheorem*{remarks}{Remarks}

\newtheoremstyle{dotless}{}{}{\itshape}{}{\bfseries}{}{ }{}

  \theoremstyle{dotless}
  \newtheorem*{thm}{}

\newtheorem*{theorem-non}{}

\numberwithin{equation}{section}
\begin{document}
\title{On Patterson's Conjecture: Sums of Quartic Exponential Sums}
\author{P. Edward Herman} 
\address{University of Chicago, Dept. of Mathematics,
5734 S. University Avenue,
Chicago, Illinois 60637}
\email{peherman@math.uchicago.edu}
\date{\today}

\begin{abstract}
     We give more evidence for Patterson's conjecture on sums of exponential sums by getting an asymptotic for a sum of quartic exponential sums over $\Q$ adjoined the eighth roots of unity. Previously, the strongest evidence of Patterson's conjecture over a number field is the paper of Livn\'{e} and Patterson \cite{LP} on sums of cubic exponential sums over $\Q[\omega], \omega^3=1.$ 
     
     The key ideas in getting such an asymptotic are a Kuznetsov-like trace formula for metaplectic forms over a quartic cover of $GL_2,$ and an identity on exponential sums relating Kloosterman sums and quartic exponential sums. To synthesize the spectral theory and the exponential sum identity, there is need for a good amount of analytic number theory.

     An unexpected aspect of the asymptotic of the sums of exponential sums is that there can be a secondary main term additional to the main term which is not predicted in Patterson's original paper \cite{P}.

      \end{abstract}
\maketitle
\section{Introduction}

The Riemann Hypothesis, Goldbach's Conjecture, Twin Primes conjecture, and Fermat's Last Theorem are all examples of lucid statements in number theory. Perhaps not as well known, Patterson's conjecture, as posed by S.J. Patterson in \cite{P}, is just as clear in its presentation.

To understand, we will motivate Patterson's conjecture. Given a polynomial $f$ with integral coefficients and a rational prime $p,$ we can ask how large the exponential sum $$\sum_{x(p)} \exp(\frac{2\pi i f(x)}{p})$$ is as $p \to \infty.$ From Weil's result on the Riemann Hypothesis for curves over a finite field, we can conclude that the sum is bounded by $O(\sqrt{p}).$ If we ask the same question for prime power or integral modulus, the ``square-root" bound continues to hold via other techniques, like stationary phase, that are much easier to apply than Weil's bound. We will call this ``square-root" bound the Weil bound whether it is for a prime, prime power, or integral modulus. 

Naturally, the next question to ask is whether or not it is possible to improve on Weil's bound on average. Specifically, fix a large $X>0$ and investigate how large the following sum of sums is: \begin{equation}\label{eq:ply}\sum_{p \leq X} \sum_{x(p)} \exp(\frac{2\pi i f(x)}{p}).\end{equation} Just using Weil's bound here and bounding the $p$-sum trivially (using the prime number theorem) would get the bound $O(\frac{X^{\frac{3}{2}}}{\log X})$ for the sum.  So we ask, can this bound be improved by allowing the exponential sums to ``cancel" with each other as they are summed over the primes less than $X?$ This seems to be a very difficult problem and generally to solve such problems already requires a deep understanding of the analogous problem of sums over positive integers. One can then hope to use a sieve or Vaughan's method to reduce to sums of primes.
Therefore, we try to improve the Weil bound as we sum over positive integers rather than over primes: \begin{equation}\label{eq:plyi}\sum_{c \leq X} \sum_{x(c)} \exp(\frac{2\pi i f(x)}{c}).\end{equation} Again the bound to beat would be $O(X^{3/2})$ for \eqref{eq:plyi}. Two examples of using a sieve on a sum of exponential sums over integers include \cite{HB} and \cite{HBP} which study the distribution of the angle associated to a cubic Gauss sum of prime modulus.  A third example is \cite{DFI} which studies the distribution of roots of a quadratic polynomial with integral coefficients modulo a prime. The sums of exponential sums over primes arise in all three examples by translating the associated distribution problems using Weyl equidistribution. We should also say in all three cases improving on the Weil bound on average was crucial to the resolution to the distribution problems. 

With the motivation now to improve upon Weil's bound on average, we state Patterson's conjecture. 
 \begin{conj}[Patterson's Conjecture] Let  $f \in \Z[x],$ with $\deg f=n>2.$ 
If we let $$S(f,X):=\sum_{c \leq X} \sum_{x(c)} \exp(\frac{2\pi if(x)}{c})$$ then for a constant $K$ depending on $f,$

\begin{equation}\label{eq:pcjj}
 S(f,X) \sim \left\{ \begin{array}{ll}
     K X^{1+2/n}   & \text{if }  f(x)= f(a-x) \text{ for some } a \in \Z; \medskip \\
         		KX^{1+1/n}   & \text{if } \text{ else}.
\end{array} \right.  
\end{equation}

\end{conj}

We note that this conjecture immediately implies an improvement over using the Weil bound and trivially summing over the integral modulus. Patterson conjectures such an asymptotic in \cite{P} with plenty of numerical experimentation and a couple examples which we now mention. In \cite{P3}, Patterson obtained the asymptotic for any $\epsilon>0,$  $$S(Ax^3,X)=K_A X^{\frac{4}{3}}+O(X^{\frac{5}{4}+\epsilon}).$$ Though we restrict to $f$ with $\deg f>2,$ we can also study $S(Ax^2+Bx+C,X)$ with $B \neq 0$ by standard methods using quadratic Gauss sums and quadratic reciprocity to show for $\delta >0,$ $$S(Ax^2+Bx+C,X)=K_{A,B,C}X^{\frac{3}{2}}+O(X^{\frac{3}{2}-\delta}).$$ Both of these examples support Patterson's conjecture. However, this seems to be the current extent of what we can prove about Patterson's conjecture over $\Z,$ so we ask about an analog of Patterson's conjecture over the integers of a number field.

Take the number field to be $\Q[\omega_3],  \omega_3^3=1.$ Let $e(Tr(z))=\exp(2\pi i(z+\overline{z})),$ and define $S(f(x),X)$ accordingly. Livn\'{e} and Patterson \cite{LP} proved that 
\begin{theorem}\label{lpe}
Suppose that $A,B,D,D' \in \Z[\omega_3]$ are such that $D,D'$ $(D,D')=1,$ both are co-prime to $3$, that any prime $(\neq \sqrt{-3})$ which divides $B$ also divides $D',$ that $27\cdot A|B^3,$ and that $\frac{B^3}{27A} \equiv \pm 1(3).$ Let $\chi$ be a Dirichlet character of modulus dividing $D'.$ Then for any $\epsilon >0,$ we have \begin{multline}
\sum_{\substack{\N(c) \leq X, (c,D')=1\\ c\equiv 0(D), c\equiv 1(3)}} S(Ax^3+Bx,X)\chi(c)=\frac{(2\pi)^{1/3}\Gamma(1/3)^2\sqrt{3}}{32\zeta_{\Q[\omega_3]}(2)} \frac{\N(D')}{\sigma(DD')}X^{\frac{4}{3}}+O(X^{\frac{5}{4}+\epsilon})\\ \mbox{ if } \chi=\left(\frac{A}{\cdot}\right)_3 \mbox{ the cubic residue symbol, and }\\ \sum_{\substack{\N(c) \leq X, (c,D')=1\\ c\equiv 0(D), c\equiv 1(3)}} S(Ax^3+Bx,X)\chi(c)=O(X^{\frac{1}{4}+\epsilon}) \mbox{ otherwise.} \mbox{ Here } \sigma(D)=\prod_{p \text{ prime}, p|D} (1+\N(p)). \end{multline}

\end{theorem}

It is not clear how this theorem fits immediately into the above conjecture for sums of exponential sums over $\Z,$ but it definitely is analogous. The incorporation of Dirichlet characters into the number field case generalizes Patterson's conjecture at the same time as solving a technical aspect of studying such sums. We will discuss more this technical need for the Dirichlet character in Section \ref{sec:skk}. 

\subsection{Main Theorem}
We now state the main theorem of the paper. Let us take $\F=\Q[\omega_8]$ the $8$-th cyclotomic field with ring of integers $\ri=\Z[\omega_8].$ We consider exponential sums over $\ri$ where the exponential sum in this case is just $\exp(2\pi i x)$ composed with the trace $Tr$ of $\F.$ Let the four complex embeddings of the field $\F$ be denoted $\eta_1, \overline{\eta_1}, \eta_2, \overline{\eta_2}.$ Let $\N(c),\phi(c)$ denote the absolute norm and Euler-Phi function of the ideal generated by $c \in \ri,$ respectively.


Let $\Psi \in C^{\infty}_0(\R^{+})$ and denote $\hat{\Psi}(s)$ as the Mellin transform over $\R,$ $$\hat{\Psi}(s)=\int_0^\infty \Psi(y)y^{s-1}dy.$$

\begin{theorem}\label{ftt0}
Suppose $A,B,F,D \in \ri=\Z[\omega_8]$ are such that  $D\equiv 1(4),$ that $B=4B'$ and that any prime $p$ dividing $B'$ also divides $D,$ and that  $\frac{B^2}{16A} \in \ri.$ Let $\theta$ be a Dirichlet character modulo $D.$ Suppose also that $A,B$ are squares. Let $\alpha$ be the best bound toward the Ramanujan conjecture for automorphic representations over $\F$ (the current best bound is $\alpha=\frac{7}{64}).$ Then for any $\epsilon > \frac{\alpha}{8},$  there exist a constants $K',K$ depending on $A,B,D$ such that

\begin{equation}
\sum_{\substack{c \in \ri,(c,D)=1\\ c\equiv 1(4)}} \frac{\prod_{j=1}^2\Psi(\frac{\sqrt{X}}{|c^{\eta_j}|^2})}{\N(c)}\theta(c)\sum_{x(c)} e(Tr(\frac{Ax^4+Bx^2+F}{c}))=\end{equation}

$$ \left\{ \begin{array}{ll}   X^{\frac{1}{2}}\hat{\Psi}(-\frac{1}{2})^2K'+ X^{\frac{1}{4}}\hat{\Psi}(-\frac{1}{4})^2K+ O(X^{\epsilon}); & \text{if } \theta\equiv \mathbf{1}(D);\medskip \\ 
X^{\frac{1}{4}}\hat{\Psi}(-\frac{1}{4})^2K+ O(X^{\epsilon}) & \text{if } \theta^4\equiv \mathbf{1}(D), \theta \not\equiv \mathbf{1} (D); \medskip \\
      O(X^{\epsilon}) & \text{if } \theta^4 \not\equiv \mathbf{1}(D),  \end{array} \right.   $$

\end{theorem}
\begin{remarks}
\begin{itemize}

\item The unnatural looking term $\prod_{j=1}^2\Psi(\frac{\sqrt{X}}{|c^{\eta_j}|^2})$ ensures that our $c$-sum is finite as well as that we sum integral elements of the field with norm up to size $X.$ Over an imaginary quadratic field we could take the nicer looking term $\Psi(\frac{X}{\N(c)}),$ however in our case an infinite number of units gets in the way of such a test function. Summing integral elements of size $|c^{\eta_j}|^2 \sim \sqrt{X}, j=\{1,2\}$ should be the correct analog of Patterson's conjecture with summing rational integers less than or equal to $X.$ These products of test functions are in fact the natural functions to put into a Bruggeman-Miatello-Kuznetsov trace formula over a number field. One can compare such products of test functions in bounding sums of Kloosterman sums over totally real fields in \cite{BMP1} and arbitary number fields in \cite{BM}.

\item The constant $K$ will depend on Fourier coefficients of metaplectic residual Eisenstein series over the quartic cover of $GL_2$ at different levels. We will make a very explicit version of Theorem \ref{ftt0} in Section \ref{sec:mint}.


\item Another vantage point of introducing the character $\theta \mod(D)$ is that we see a kind of ``inner product" in Theorem \ref{ftt0} by averaging over $c$ of $\theta$ against the sum quartic exponential sum. That the inner product is large if $\theta^4\equiv \mathbf{1}\mod(D),$ and small for any other $\theta$ would be a facet we would like to see in a more general Patterson's conjecture.

\item Normally to apply the theory of quartic metaplectic forms one only needs the field to contain the fourth roots of unity. This is almost true in this case except that we need $2$ to be a square which requires the eighth roots of unity. We will discuss this more in Section \ref{sec:skk}.

\end{itemize}
\end{remarks}

We consider Theorem \ref{ftt0} as a smoothed version of Patterson's conjecture. With more careful analysis, the smoothing could be removed with loss in the error term. We also normalized by a $\frac{1}{\N(c)}$-factor to connect such sums more easily to the spectral theory of metaplectic forms which are crucial to the result. A summation by parts argument would remove this normalization if desired. 

Recall from Patterson's conjecture that if $f(a-x)=f(x)$ then we expect different behavior for $S(f(x),X),$ namely it is asymptotic to $X^{1+\frac{2}{n}}$ instead of $X^{1+\frac{1}{n}}.$ In our Theorem \ref{ftt0}, $f(x)=Ax^4+Bx^2+F=f(-x),$ and so we do see in the case $\theta$ is the trivial character, that our main term is is asymptotic to $X^{\frac{1}{2}}=X^{\frac{2}{4}},$ which is correct with our normalization of the $c$-sum. Believing that Patterson's conjecture would generalize to number fields, in that analogous sums would have analogous main terms in $X,$ Livn\'{e} and Patterson's Theorem \ref{lpe} as well as our Theorem \ref{ftt0} are in agreement. What was not expected is that also in these two choices of $\theta,$ there is a secondary main term. It is not clear if such secondary terms exists from the data in \cite{P} for exponential sums over $\Z.$

Another feature of our theorem is that if $\theta$ is the quartic residue character or a quadratic character, then the main term of size $X^{\frac{1}{2}}$ vanishes, and we are left with a main term of size $X^{\frac{1}{4}},$ which is what we would expect in Patterson's conjecture from a polynomial $f$ having no $a\in \ri$ such that $f(a-x)=f(x).$ This sort of property tells us that a correct version of Patterson's conjecture for number fields will be much more delicate to state than for $\Q.$ 

If $\theta$ is a character not of order $4,$ then our theorem looks much similar to Theorem \ref{lpe} for the case $\chi^3 \neq 1.$ Our error term is much better due to the smoothing as well as using the strongest results known towards the Ramanujan conjecture for automorphic forms, which via the Shimura correspondence gives strong bounds for metaplectic forms.

\section{Outline of Proof}\label{sec:skk}

An important role in our theorem is played by sums of Kloosterman sums. In particular, asking for a similar asymptotic to Patterson's conjecture for a sum of Kloosterman sums. Such asymptotics have been studied using the spectral theory of automorphic forms originally by Kuznetsov \cite{Kuz}; Goldfeld and Sarnak \cite{GS}; and Deshouillers and Iwaniec \cite{DeshI}.  This connection between sums of Kloosterman sums and automorphic forms is well known and is seen through the Kuznetsov or relative trace formula. In its most trivial form this connection can be seen for a $V \in C^{\infty}_0(\R^{+})$ for large $X>0$ as \begin{equation}\label{eq:iwk}\sum_{c=1}^\infty \frac{S(n,m,c)}{c} V(\frac{X}{c})=\sum_{\substack{\pi \\ \frac{1}{2} < s_{\pi}<1}} X^{s_{\pi}-\frac{1}{2}}\tilde{V}(s_{\pi})a_{\pi}(n)\overline{a_{\pi}(m)}+O(X^{\epsilon})\end{equation} for any $\epsilon>0$ and $\tilde{f}$ some Bessel transform of $V.$ The Kloosterman sum here is $S(n,m,c)=\sum_{x(c)^{*}}e(\frac{mx+n\overline{x}}{c}).$ The sum over $\pi$ is of automorphic forms of level one with eigenvalue $s_{\pi}(1-s_{\pi}).$ Forms with eigenvalues in this range $\frac{1}{2} < s_{\pi}<1$ are called exceptional. In this special case of level one forms over $\Q$, there are no exceptional forms (see \cite{DeshI}) and this sum of Kloosterman sums has no asymptotic but just a bound $O(X^{\epsilon})$ for any $\epsilon>0.$ The non-existence of such exceptional forms for general congruence subgroups is
the Selberg eigenvalue conjecture, which is intimately tied to the Ramanujan conjecture for automorphic representations. 

Rather than look at look at automorphic forms over $\Q$ we look at forms over $\F.$ As well, we do not look quite at automorphic forms but at metaplectic forms on $\F.$ These are functions on $\Gamma \backslash SL_2(\C) \slash SU(2) \times \Gamma \backslash SL_2(\C) \slash SU(2),$ for a certain discrete subgroup $\Gamma,$ that--at their simplest-- transform on the left by the Kubota residue symbol. We will be more explicit about their definition in Section \ref{sec:meta}. Also, analogous to automorphic forms, there is an associated spectral theory of metaplectic forms and a similar connection to sums of Kloosterman sums. However, in this case we know a metaplectic form with an exceptional parameter exists, namely the quartic theta function which is a residual Eisenstein series. This was discovered by Kubota in \cite{Ku}. Label the residual Eisenstein series $f_{00}$ with eigenvalue parameter $s_{00}.$ Assume the field we study is just $\Q[i]$ a subfield of $\F.$ Then the quartic metaplectic analog of \eqref{eq:iwk} in its simplest form is \begin{equation}\label{eq:iwm}\sum_{c \in \Z[i],c\equiv 1(4)}^\infty \frac{S_4(n,m,c)}{\N(c)} V(\frac{X}{\N(c)})= X^{s_{00}-1}\tilde{V}(s_{00})a_{f_{00}}(n)\overline{a_{f_{00}}(m)}+ \sum_{\substack{\pi \\ 0 < s_{\pi}<s_{00}}} X^{s_{\pi}-1}\tilde{V}(s_{\pi})a_{\pi}(n)\overline{a_{\pi}(m)}+O(X^{\epsilon}),\end{equation} with $$S_4(n,m,c)=\sum_{x(c)^{*}}\left(\frac{x}{c}\right)_4 e(Tr(\frac{mx+n\overline{x}}{c})).$$ 

So how does this sum of metaplectic Kloosterman sums relate to a sum of quartic exponential sums in Theorem \ref{ftt0}? As mentioned in the introduction, there are connections or identities between Kloosterman sums and other exponential sums. The exponential sum for a prime modulus in Theorem \ref{ftt0} (without the term $e(\frac{F}{p})$) can be written for $p$ prime as $$\sum_{x(p)} e(Tr(\frac{Ax^4+Bx^2}{p}))=\sum_{x(p)} [(\frac{x}{p})_2+1] e(Tr (\frac{Ax^2+Bx}{p})).$$ Our interests is in summing the left hand side over a general $c \in \ri.$ On the right hand side if we consider a sum $\sum_{x(c)}e(Tr(\frac{Ax^2+Bx}{c}))$  over $c \in \ri$ as we do with Kloosterman sums in \eqref{eq:iwm} it is straightforward to understand its asymptotic. We study such a sum in Section \ref{sec:qdd}.  However, if we took a sum over $c\in \ri$ of the sum $\sum_{x(c)}(\frac{x}{c})_2 e(Tr(\frac{Ax^2+Bx}{c}))$, then the asymptotic is not straightforward. 
In Section \ref{sec:qdd}, we prove --only assuming the field contains the fourth roots of unity-- that for $p$ prime  \begin{equation}\label{eq:pqpq}\sum_{x(p)} \left(\frac{x}{p}\right)_2 e(Tr(\frac{Ax^2+Bx}{p}))=\left(\frac{AB}{p}\right)_2 e(Tr(\frac{-B^2\overline{8 A}}{p}))\sum_{b(p)} \left(\frac{b}{p}\right)_4 e(Tr(\frac{\overline{4^2A}B^2(b+\overline{b})}{p})).\end{equation} Likewise, assuming the field contains the eighth roots of unity we prove for a general integral element $c\equiv 1(4),(AB,c)=1,$ \begin{multline} \label{eq:qbo} \sum_{x(c)}e(Tr(\frac{Ax^4+Bx^2}{c}))=\left(\frac{AB}{c}\right)_2 e(Tr(\frac{-B^2\overline{8 A}}{c}))S_4(\overline{4^2A}B^2,\overline{4^2A}B^2,c)+\sum_{x(c)}e(Tr(\frac{Ax^2+Bx}{c}))+\\ \{ \mbox{Cross terms}\}.\end{multline} Ignoring the ``Cross terms" for a moment we have replaced the quartic sum over an integral element $c \in \ri$ with a Kloosterman sum and an elementary quadratic sum.

\begin{remark}Why do we need the fourth roots of unity for the identity over a finite field but the eighth roots of unity for an element of the ring of integers? When one proves the identity for prime powers $p^m$ in Section \ref{ppoww} it is crucial to understand not only solutions $Ax^4+Bx^2+F \equiv 0 \mod(p^m)$ but solutions of its derivative $4Ax^3+2Bx \equiv 0 \mod(p^m).$ Assuming $x \not\equiv 0(p^m),$ for \eqref{eq:qbo} to hold for a prime power the latter solutions of the derivative must have a non-trivial solution to $x^2 \equiv B\overline{2A} \mod(p^m).$ More so, this must be true uniformly for all $p$ coprime to $2.$ This clearly requires $2$ to be a global square which is the case in $\F=\Q[\omega_8].$ If we compare this same idea of going from the finite field case to the integral case in the Katz-Livn\'{e}-Patterson cubic identity there is a minor ``miracle" that $3$ is square in $\Z[\omega_3]$ and therefore they do not need to extend the field from the minimal field needed to study cubic metaplectic forms.\end{remark}


Consider summing \eqref{eq:qbo} over $c \in \ri$. If we can ``swap out" the metaplectic Kloosterman sum with a spectral sum as in \eqref{eq:iwm}, then we will be able to get an asymptotic using the fact that a residual Eisenstein series exists. In essence, that is what is done in \cite{LP} for sums of cubic exponential sums over $\Q[\omega_3].$
There the connection between Kloosterman sums and exponential sums is simpler and seen in the identity for $c \equiv 1(3),$ \begin{equation}\label{eq:kbo} \sum_{x(c)} e(Tr(\frac{Ax^3+Bx}{c}))=(\frac{\overline{A}}{c})_3 \sum_{y(c)^{*}} (\frac{y}{c})_3 e(Tr(\frac{y-B^3\overline{3^3A y}}{c})).\end{equation}

Assuming the Kloosterman sum has an asymptotic from the residual Eisenstein series and the sum of quadratic exponential sums is understood from Section \ref{sec:qdd}, we must understand the ``Cross-terms" when summing over $c \in \ri.$ We sketch how to address this issue in Section \ref{sec:saq}.

\subsection{Archimedean analog of exponential sums}

The cubic sum we just mentioned has an analogous identity from the archimedean perspective which is called Nicholson's identity: $$\int_0^\infty \cos(x^3+Bx)dx=\frac{B^{1/3}}{3}K_{1/3}(2(B/3)^{3/2}).$$ The connection between the exponential sum and archimedean identity is more readily seen when rewritten (see \cite{DI}) as $$\int_{-\infty}^\infty \exp(i(Ax^3+x))dx=\frac{1}{\sqrt{3}}\int_0^\infty (x/A)^{1/3}\exp(-x-(3^3Ax))^{-1}\frac{dx}{x}.$$ The archimedean analog of the quartic identity of \eqref{eq:pqpq} is $$\int_0^\infty \cos(Ax^2) \sin(Bx)\frac{dx}{\sqrt{x}}=\frac{\sqrt{B}}{2\sqrt{A}}\cos(\frac{B^2}{8A}-\frac{3\pi}{8})J_{1/4}(\frac{B^2}{8A}),$$ using (\cite{GR},6.686) and $J_{1/2}(Bx)=\sin(Bx)\sqrt{\frac{2}{\pi x}}.$

Our simplified explanation above, for understanding the asymptotic for certain sums of quartic exponential sums via a spectral theory of metaplectic forms, assumes virtually no technicalities. The problem is that there are many difficulties in the transition between the sums of Kloosterman sums and quartic exponential sums, and we explain those now.

\subsection{Choice of trace formula}
The first technical point is that we clearly need a trace formula for metaplectic forms. It turns out we can adapt a trace formula for automorphic forms over an number field for our metaplectic spectrum. Let us mention the many trace formulas for automorphic forms we can choose from. In \cite{BM}, Bruggeman and Miatello create a Kuznetsov trace formula over any number field, but they choose the two Poincare series needed in the construction of the formula both at the cusp $\infty.$ In \cite{BMo} and \cite{L}, Bruggeman and Motohashi and Lokvenec-Guleska, respectively, construct a trace formula over an imaginary quadratic field with representations of all weights and expansions at different cusps, but the test functions involved--since they include all representations of various weights--are difficult to analyze. We mention Louvel in \cite{Lo} does a similar analysis to what we need in this paper using the trace formula of \cite{L}. However, his analysis of test functions on the spectral side of the trace formula is not strong enough for the results we need. This seems to be due to the extra representations of non-zero weight. In fact since we are only concerned with the residual Eisenstein series of weight $0,$ having the extra representations of non-zero weight is unnecessary. Another formula is due to Bruggeman, Miatello, and Pacharoni \cite{BMP1} where they construct a trace formula over a totally real field with different cusps; it is very similar to \cite{BM} in design. 

We chose to create our own trace formula in the appendix that is a direct consequence of combining the ideas of \cite{BM} over an number field and \cite{BMP1} which uses various cusps. The advantage of this Bruggeman-Miatello Kuznetsov trace formula is that it is directly amenable to the estimates of the papers of Bruggeman-Miatello \cite{BM1} (respectively, Miatello-Wallach \cite{MW}). There the authors take their trace formula over a real rank one group (resp. product of real rank one groups) and give asymptotics for sums of Kloosterman sums. The key to getting such estimates, which we describe in detail in Section \ref{sec:mww}, is that by choosing the test function $f$ on the spectral side of the trace formula, we deal with a Bessel transform of $f$ on the geometric side of the trace formula. By careful analysis following \cite{BM1}, we can realize this Bessel transform as a much simpler Mellin transform of $f$ plus a negligible error term. 

\subsection{Choice of cusps}

The reason we need a Kuznetsov trace formula with freedom to choose two different cusps is to synthesize the connection from Kloosterman sums to quartic exponential sums via the identity \eqref{eq:qbo}. We note from \eqref{eq:qbo} we looked at the ``unramified" case where  $(AB,c)=1,$ but this does not take care of all cases. Livn\'{e} and Patterson \cite{LP} seeing the same problem, cleverly have all primes that divide $B$ divide the level $D$ of the metaplectic forms connected to the Kloosterman sum in question. Then looking at the trace formula at specific cusps, the $c$-sum will be co-prime to $D,$ so ``unramified" cases are all the cases and there are no cases where $(AB,c)>1.$ They call this a technical point that needs to be removed for a more general statement, but  in both theirs and our case, it will require much deeper local analysis of these kind of exponential sum identities in \eqref{eq:qbo} and  \eqref{eq:kbo} in ``ramified" situations.

The cusps chosen depend on the metaplectic forms with nontrivial level $D$ and some nebentypus $\theta \mod (D).$ The concern then is, does there exists a residual Eisenstein series over the quartic cover of $GL_2$ that transforms by not only the quartic residue symbol but also by this nebentypus character $\theta$ of modulus $D?$ The work of Kubota \cite{Ku} or Patterson's work on the cubic cover does not cover these situations. So in Section \ref{sec:err} we show that such a residual Eisenstein series exists.

\subsection{The difficult ``Cross terms"}\label{sec:saq}
If we accept that the sums of Kloosterman sums and sums of quadratic exponential sums are understood from \eqref{eq:qbo}, we are left with the asymptotics of terms we call ``Cross terms" which look like  $$\sum_{\substack{n,m \in \ri\\ (n,m)=1\\ nm\equiv 1(4)}} \frac{\theta_D(nm)\prod_{j=1}^2\Psi(\frac{\sqrt{X}}{|(nm)^{\eta_j}|})}{\N(nm)} \left[\sum_{x(m)}e(\frac{\overline{n}(Ax^2+Bx)}{m})\right]\left[e(\frac{-B^2\overline{8 Am}}{n})S_4(-B^2\overline{m4^2A},n)\right].$$ In Section \ref{sec:tro} we show that by fixing the $m$-sum, the $n$-sum can be connected to another spectral sum of metaplectic forms of level $4Dm^2.$ This will give another main term of size $X^{1/4}.$ In other words, for each $m$ we will get a main term corresponding to a residual Eisenstein series  of level $4Dm^2$ at cusps $\{\sigma_a,\sigma_b\}$ (different than the cusps used for the other main term) of size $X^{1/4}.$ We state Theorem \ref{crrm} now and prove it in Section \ref{sec:tro}. We give some notation first. Let $f_{00,4Dm^2}$ denote the residual Eisenstein series of level $4Dm^2,$ and transforming by character $\theta_D \chi   (\frac{\cdot}{m})_2 \mod (4Dm^2)$ with $\chi$ a Dirichlet character modulo $m^2.$ The residual Eisenstein series has $r$-th Fourier coefficient at the cusp $\sigma$ denoted by $c_{r,\sigma}(f_{00,4Dm^2}),$ and  $d_{r,\sigma}$ is a certain normalization needed in the trace formula. We have

\begin{theorem}\label{crrm}
 
  Assume the same hypothesis and notation as Theorem \ref{ftt0}. For a constant $K_{\frac{DB^2}{16A},\frac{B^2}{16A}}$ depending on certain cusps $\sigma_a,\sigma_b$ defined in Theorem \ref{ftt}, we have \begin{gather*}\label{eq:harde} \sum_{\substack{n,m \in \ri\\ (n,m)=1\\ nm\equiv 1(4)}} \frac{\theta_D(nm)\prod_{j=1}^2\Psi(\frac{\sqrt{X}}{|(nm)^{\eta_j}|^2})}{\N(nm)} \left[\sum_{x(m)}e(\frac{\overline{n}(Ax^2+Bx)}{m})\right]\left[e(\frac{-B^2\overline{8 Am}}{n})S_4(-B^2\overline{m4^2A},n)\right]=\end{gather*}

$$ \left\{ \begin{array}{ll}   
X^{\frac{1}{4}}\hat{\Psi}(-\frac{1}{4})^2K_{\frac{DB^2}{16A},\frac{B^2}{16A}}+ O(X^{\epsilon}) & \text{if } \theta_D^4\equiv \mathbf{1}(D); \medskip \\
      O(X^{\epsilon}) & \text{if } \theta_D^4 \not\equiv \mathbf{1}(D).  \end{array} \right.   $$

The term $K_{\frac{DB^2}{16A},\frac{B^2}{16A}}$ is a convergent sum \begin{gather*}\label{eq:wconr}\frac{1}{g(\frac{-1}{4})^2}(\frac{\pi}{2})^2 \bigg\{ \sum_{\substack{m \in \ri\\ m\equiv 1(4)}} \frac{\theta_D( m)\N(m)^{1/4}}{\phi(m)} \sum_{\substack{\chi \mod (m)\\ \chi^4 \equiv 1\mod (m)}} \tau(\chi) \chi(8A(\overline{-B^2})) \times \\ \bigg[\overline{c_{\frac{4DB^2}{16A},\sigma_a}(f_{00})}c_{\frac{B^2}{16A},\sigma_b}(f_{00})d_{\frac{4DB^2}{16A},\sigma_a}(\frac{1}{4})\overline{d_{\frac{B^2}{16A},\sigma_b}(\frac{1}{4})}+\frac{\{\mbox{Remainder}\}}{X^{\frac{1}{4}} \hat{\Psi}(\frac{-1}{4})}\bigg]\bigg\}\end{gather*}  Here $\tau(\chi)$ is a Gauss sum. The Remainder term is defined in \eqref{eq:ress}, it is just the spectral sum of metaplectic forms excluding the residual Eisenstein series. In terms of $X,$ the $\{Remainder\}$ term is $o(X^{1/4}).$ 
\end{theorem}

\subsection{Appendix}
As we mentioned earlier the first part of the Appendix is concerned with a construction of a Kuznetsov trace formula over $\F$ with multiple cusps.

The second part of the appendix is more of a commentary on how far we can extend the exponential sum identities like \eqref{eq:qbo} in connecting exponential sums with integral polynomial arguments and $GL_2$ Kloosterman sums. It seems likely that any higher degree polynomials will not be associated to rank 2 Kloosterman sums.

{\bf Acknowledgements.} The author would like to thank  Samuel Patterson and Akshay Venkatesh for pointing out errors in earlier drafts of the paper.

\section{Preliminaries}

\subsection{Notation}

Let $\F=\Q[\omega_8]$ and $\ri=\Z[\omega_8]$ be its ring of integers. Define $\mathbb{G}=R_{\F\slash \Q} (SL_2)$ as the restriction of scalars of $SL_2$ over $\F.$ 
The field has $4$ complex embeddings $\eta_1, \overline{\eta_1}, \eta_2, \overline{\eta_2}.$ We define $$G:=\mathbb{G}(\R) \simeq SL_2(\C) \times SL_2(\C)$$ and $$\mathbb{G}(\Q) \simeq \{(x^{\eta_1}, x^{\eta_2}): x \in SL_2(\F)\}.$$ We consider $\F$ as embedded in $\C^2$ by $\psi \in \F \to (\psi^{\eta_1}, \psi^{\eta_2}),$ and the image of $SL_2(\ri)$ corresponds to $\mathbb{G}(\Z).$ Let $K$ be the maximal compact subgroup $SU_2(\C)\times SU_2(\C)$ of $G = \mathbb{G}(\R).$

For $x=(x_1,x_2)\in \C^2,$ we let $S(x):=2[\Re(x_1)+\Re(x_2)]$ which extends the trace $Tr_{\F \slash \Q} :\F \to \Q.$ As well, we let $\N(x)=|x_1|^2|x_2|^2$ extend the norm of $\F$ over $\Q$ to $\C^2.$

\subsubsection{Functions of product type}

The test functions on $G$ that we use are of product type: $f(g)=f_1(g_1)f(g_2)$ for $g=(g_1,g_2) \in G$ with $f_j$ a complex valued function on $SL_2(\C).$

\subsubsection{Subgroups of $G$} For $y \in \R^{+,2}$ we put $$a[y]:=\left( \begin{pmatrix} \sqrt{y_1} & 0 \\ 0 & \frac{1}{\sqrt{y_1}}\end{pmatrix},\begin{pmatrix} \sqrt{y_2} & 0 \\ 0 & \frac{1}{\sqrt{y_2}}\end{pmatrix}\right).$$ This is the identity component of a maximal $\R$-split torus in $G$ which we label $A.$ We normalize the Haar measure of $A$ by $da=\frac{dy_1}{y_1}\frac{dy_2}{y_2}.$

For $x \in \C^2,$ we let $$n[x]:=\left(\begin{pmatrix}1 & x_1 \\ 0& 1 \end{pmatrix},\begin{pmatrix}1 & x_2 \\ 0& 1 \end{pmatrix}\right) \in G.$$ The normalization of the Haar measure for $N:=\{n[x]:x \in \C^2\}$ is $\frac{dx_1\overline{dx_1}}{-2\pi i}\frac{dx_2\overline{dx_2}}{-2\pi i}.$

For $u \in \C^{*,2}$ we define $$b[u]:=\left( \begin{pmatrix}u_1 & 0 \\ 0& \frac{1}{u_1}\end{pmatrix},\begin{pmatrix}u_2 & 0 \\ 0& \frac{1}{u_2}\end{pmatrix}\right).$$ Let $M$ be the subgroup $\{b[u]:|u_j|=1\}$ of $K=SU(2) \times SU(2).$ We then have the Iwasawa decomposition $G=NAK$ with standard Parabolic subgroup $P:=NAM.$

\subsection{Cusps}

Let $\Gamma$ be a finite index subgroup of $SL_2(\ri),$ then it follows that $\Gamma$ has a finite number of cusp classes. Let $\mathcal{P}$ be a set of representatives of those classes. For each $\sigma \in \mathcal{P},$ we fix $g_{\sigma}$ such that $\sigma=g_{\sigma} \cdot \infty.$

For each $\sigma \in \mathcal{P}$ there is a parabolic subgroup $P^{\sigma}:=g_{\sigma}Pg_{\sigma}^{-1},$ with decomposition $P^{\sigma}=N^{\sigma}A^{\sigma}M^{\sigma},$ with $N^{\sigma}:=g_{\sigma}Ng_{\sigma}^{-1}, A^{\sigma}=g_{\sigma}Ag_{\sigma}^{-1},$ and $M^{\sigma}=g_{\sigma}Mg_{\sigma}^{-1}.$ We use conjugation by $g_{\sigma}$ to transport Haar measures of $N,A,$ and $M$ to $N^{\sigma},A^{\sigma},$ and $M^{\sigma}.$

For each cusp $\sigma \in \mathcal{P}$ and $g \in G$ we have unique decomposition $g=n_{\sigma}[g]g_{\sigma}a_{\sigma}[g]k_{\sigma}[g]$ with $n_{\sigma}[g] \in N^{\sigma}, a_{\sigma}[g]\in A,$ and $k_{\sigma}[g] \in K.$ 

We choose $a^{-2\rho} dn da$ as the left invariant measure on $G\slash K,$ where we note $a[y]^{\rho}=\sqrt{\N(y)}.$

\subsection{Discrete subgroups}
Let $\Gamma_{P^{\sigma}}:=\Gamma \cap P^{\sigma}$ and $\Gamma_{N^{\sigma}}:=\Gamma \cap N^{\sigma}.$ There is a lattice $t_{\sigma}$ such that $\Gamma_{N^{\sigma}}=\{g_{\sigma}n[\xi]g_{\sigma}^{-1}:\xi \in t_{\sigma}.\}$

\subsection{Characters} \label{sec:chcc}
We put $t_{\sigma}'=\{r \in \F: Tr(rx) \in \Z \mbox{ for all } x \in t_{\sigma}\}.$
 All characters of $N^{\sigma}$ are $\chi_r:g_{\sigma}n[x]g_{\sigma}^{-1} \to e^{2\pi i S(rx)}$ for $r \in \C.$ All characters of $\Gamma_{N^{\sigma}}\backslash N^{\sigma} $ are obtained by taking $r \in t_{\sigma}'.$  For a character $\theta$ of $\Gamma,$ we say a cusp $\sigma$ is {\it essentially cuspidal} if  $\theta\vert_{\Gamma_{N^{\sigma}}}=1$.

The important subgroups of ${\rm SL}_2(\ri)$ for us are defined by

\begin{align}
\Gamma_1(4) \ \
&=\{\gamma\in {\rm SL}_2(\ri) \, : \, \gamma\congru I \pmod 4\}, \label{Kloo:eq:3}\\
\Gamma_0(D) 
&= \{\gamma\in {\rm SL}_2(\ri) \, : \, \gamma\congru 
\left(\begin{smallmatrix} * & * \\ 0 & * \end{smallmatrix}\right)  \pmod D\}. \label{Kloo:eq:4}
\end{align}
The Kubota symbol $\kappa$ can now be introduced. It is defined on $\Gamma_1(4)$ by

\begin{equation*}
\kappa (\gamma) = \begin{cases} 
\left(\frac{a}{c}\right)_{\!4} & \ec{if $c\neq 0$}\\
1 &\ec{if $c=0$},
\end{cases} \qquad \ec{ where } \gamma=\begin{pmatrix}a&b\\ c&d\end{pmatrix}\in \Gamma_1,
\end{equation*} with $\left(\frac{a}{c}\right)_{\!4}$ the quartic power residue symbol.

We let $\Gamma_2$ be the group generated by $\Gamma_1(4)$ and $SL_2(\Z).$ It was shown in [\cite{P4}, p. 127] that the cubic residue character extends to the group generated by $\Gamma_1(3)$ and $SL_2(\Z)$ for the ring $\Z[\omega_3]$ if it is defined to be trivial on $SL_2 (\Z).$ It is an analogous argument to extend $\kappa$ in this setting to $\Gamma_2$ by also defining it to be trivial on $SL_2(\Z).$ 

\section{Metaplectic forms}\label{sec:meta}

In the rest of this section we shall consider the discrete subgroup of $SL_2(\ri),$ $\Gamma=\Gamma_1(4) \cap \Gamma_0(D), D \in \ri.$  Let $\theta$ be a Dirichlet character modulo $D$ over $\ri.$ We consider $\theta$ also as a character of $\Gamma$ by $$\theta(\begin{pmatrix}a&b\\ c&d\end{pmatrix})=\theta(d).$$
Let $L^2\left( \Gamma \backslash G,\theta \kappa\right)$ be the space of square-integrable functions on $G$ which satisfy 

\begin{equation*}
f(\gamma gk) = \theta(\gamma)\kappa(\gamma) f(g) \qquad \ec{for all } \gamma \in \Gamma, k \in K.
\end{equation*} 
Under the action of $G$ by the regular representation, $L^2\left( \Gamma \backslash G,\theta\kappa\right)$ decomposes into irreducible unitary representations with finite multiplicities, say 

\begin{equation}
L^2\left( \Gamma \backslash G,\theta \kappa\right)=\barre{\bigoplus V}.
\end{equation}

Let us choose a complete orthonormal basis $\{f_l\}_{l \geq 0}$ of the space of $K$-finite vectors in  the discrete spectrum, call it $L_d^2\left( \Gamma \backslash G,\theta \kappa\right),$ of $L^2\left( \Gamma \backslash G,\theta \kappa\right)$ such that each $f_l$ generates one of the irreducible $V$ under the action of $G.$ The orthogonal complement of $L_d^2\left( \Gamma \backslash G,\theta \kappa\right)$ in $L^2\left( \Gamma \backslash G,\theta \kappa\right),$ call it $L_c^2\left( \Gamma \backslash G,\theta \kappa\right),$ is described by integrals of unitary Eisenstein series.  

As $G\slash K=SL_2(\C)\slash SU_2(\C) \times SL_2(\C)\slash SU_2(\C),$ each $f_l \in L_d^2\left( \Gamma \backslash G,\theta \kappa\right)$ is an eigenvector of the Laplace operators $L_1,L_2.$ We standardly normalize the Laplace operator to be $L_jf_l=\lambda_{l,j} f_l,$ with $\lambda_{l,j}=1-\mu_{l,j}^2, \mu_{l,j} \in \left(i[0,\infty) \cup (0,1]\right).$ Forms with $\mu_{l,j} \in (0,1]$  for any $j$ are said to have {\it exceptional spectral parameter.}

\subsection{Fourier Expansion of the discrete spectrum}
Define $$W_{\sigma}^{r,\nu}(ng_{\sigma}a[y]k):=\chi_r(n)\sqrt{N(y)}\prod_{j=1}^2 K_{\nu_j}(4\pi y_j\sqrt{|r_j|}),$$ with $K_{\nu}$ the $K$-Bessel function of index $\nu.$ Let $$d_{r,\sigma}(\nu):=
\prod_{j=1}^2 d_{r,j}(\nu_j)$$ with 
$$d_{r,\sigma}(\nu_j)=\frac{1}{vol(\Gamma_{N^{\sigma}}\backslash N^{\sigma})} \frac{2^{1-\nu_j}(2\pi\sqrt{|r_j|})^{-\nu_j}}{\Gamma(\nu_j+1)}.$$  Fix a cusp $g_{\sigma}\cdot \infty$ which is essentially cuspidal with respect to $\kappa, \theta.$ Then a metaplectic form $f_l$ is invariant under $\Gamma_{N^{\sigma}}$ and therefore has a Fourier expansion at the cusp $\sigma=g_{\sigma}\cdot \infty.$ The Fourier expansion at the cusp $\sigma$ is $$f_l(g)=F_{0,\sigma}f_l(a_{\sigma}[g])+\sum_{r \in t_{\sigma}'} c_{r,\sigma}(f_l)d_{r,\sigma}(\mu_f)W_{\sigma}^{r,\mu_l}(g),$$ with $F_{0,\sigma}f_l(a_{\sigma}[g])$ the Fourier term at $r=0.$ It is not crucial to explicate the zeroth Fourier coefficient, but if a form $f_l$ has a vanishing zeroth Fourier coefficient for all cusps $g_{\sigma}\cdot \infty$ of $\Gamma \backslash G,$ we call the form a {\it metaplectic cusp form} or a {\it cusp form}.
\subsection{Eisenstein Series}
We now describe $L_c^2\left( \Gamma \backslash G,\theta \kappa\right).$ For each $\sigma \in \mathcal{P},$ there is an Eisenstein series $$E(P^{\sigma},\nu, i\mu,g):=\sum_{\gamma \in \Gamma_{P^{\sigma}}\slash \Gamma} a_{\sigma}[\gamma g]^{\rho+2\nu\rho+i\mu}.$$ Here $\nu \in \C,$ and $\mu$ is an element of a lattice in the hyperplane $\Re(x)=0, x \in \C^2.$ In particular, $\mu$ is defined by $a[\gamma]^{\mu}=1$ for $\gamma \in \Gamma_{P^{\sigma}},$ again using similar notation to \cite{BM}. The series converges for $\nu > \frac{1}{2},$ and has meromorphic continuation in $\nu$ with Laplaican eigenvalue $\bigg(1- (\nu+\mu_1)^2,1- (\nu+\mu_2)^2\bigg).$

There is a Fourier expansion at a cusp $\sigma'$ for $E(P^{\sigma},\nu, i\mu,g)$ similar to the discrete spectrum. We follow \cite{Ku}. 

\begin{prop} Denote by $\mathbf{e}_{\F,R}$ the group of units of $\F$ congruent to $1$ modulo a natural number $R.$ Then for a suitable $R$ there exists a natural number $k_{\sigma'}$ such that \begin{multline} E(P^{\sigma},\nu,i\mu,na[z] k)=F_{0,\sigma,\sigma'}(\nu+i\mu,a[z])+\\ \sum_{r \in t_{\sigma}'} D_{r,\sigma}(P^{\sigma},\nu,i\mu)d_{r,\sigma}(\nu+i\mu)W_{\nu+i\mu,r}(\sqrt{\N(z)}),\end{multline} with 
$$F_{0,\sigma,\sigma'}(s,a[z])=\frac{1}{vol(\Gamma_{N^{\sigma'}}\backslash \Gamma)}\psi_{\sigma,\sigma'}(0,s)\frac{\pi}{s},$$  $$D_{r,\sigma}(P^{\sigma},\nu,i\mu)=\frac{1}{vol(\Gamma_{N^{\sigma'}}\backslash \Gamma)}\psi_{\sigma,\sigma'}(r,\nu+i\mu),$$
and $$\psi_{\sigma,\sigma'}(r,s)=\frac{1}{k_{\sigma}} \sum_{c \in \mathbf{e}_{\F,R}\backslash{}^{\sigma'}\mathcal{C}^{\sigma}} \frac{1}{\N(c)^{s}} \sum_{\left( \begin{array}{cc}
a &b \\ c & d \\ \end{array} \right) \in 
{}^{\sigma'}\mathcal{L}(c)^{\sigma}} \left(\theta\kappa( \begin{array}{cc}
a &b \\ c & d \\ \end{array} )\right)^{-1} e(-Tr(\frac{rd}{c})).$$

Here ${}^{\sigma'}\mathcal{C}^{\sigma}$ and ${}^{\sigma'}\mathcal{L}(c)^{\sigma}$ are defined in the next section, specifically in Definition \ref{coy} and Proposition \ref{coy1}, respectively. 

\end{prop}

\begin{proof}
This follows immediately from Theorem $5$ and equations $(38)-(42)$ of \cite{Ku}.
\end{proof}

One feature of metaplectic forms that is not so readily seen in automorphic forms is that there exists a discrete non-cuspidal metaplectic form. According to Selberg's theory, 
 such forms are residues of Eisenstein series. We will discuss this more in Section \ref{sec:err}.

\section{Bruggeman-Miatello-Kuznetsov Trace Formula over $\F$ with multiple cusps}

In order to state the trace formula we need to define some terms. Let $s_0:=\begin{pmatrix}0&-1 \\ 1&0\end{pmatrix}.$ Then the Bruhat decomposition for $SL_2(\F)$ is a union of $SL_2(\F)\cap P$ and the big cell $C:=(P\cap SL_2(\F))s_0(N\cap SL_2(\F)).$

\begin{definition}\label{coy}
Let $\sigma,\sigma' \in \mathcal{P}.$ We define ${}^{\sigma'}\Gamma^{\sigma}:=\Gamma \cap g_{\sigma'}Cg_{\sigma}^{-1}$ and $${}^{\sigma'}\mathcal{C}^{\sigma}:=\{c \in \F^{*}:g_{\sigma'}^{-1}\gamma g_{\sigma}=\begin{pmatrix}\cdot&\cdot \\ c&\cdot \end{pmatrix} \mbox{ for some } \gamma \in {}^{\sigma'}\Gamma^{\sigma} \}.$$

For each $c \in {}^{\sigma'}\mathcal{C}^{\sigma}$ we put ${}^{\sigma'}\Gamma(c)^{\sigma}:=\{ \gamma \in {}^{\sigma'}\Gamma^{\sigma}: g_{\sigma'}^{-1}\gamma g_{\sigma}=\begin{pmatrix}\cdot&\cdot \\ c&\cdot \end{pmatrix}\}.$
\end{definition}

\begin{prop}\label{coy1}Let $\sigma,\sigma' \in \mathcal{P}.$ For each $c \in {}^{\sigma'}\mathcal{C}^{\sigma}$ there is a finite set ${}^{\sigma'}\mathcal{L}(c)^{\sigma}$ such that $${}^{\sigma'}\Gamma(c)^{\sigma}= \cup_{\gamma \in {}^{\sigma'}\mathcal{L}(c)^{\sigma}} \Gamma_{N^{\sigma'}}\gamma \Gamma_{N^{\sigma}}.$$ The set $${}^{\sigma'}\mathcal{L}^{\sigma}:=\cup_{c \in {}^{\sigma'}\mathcal{C}^{\sigma}} {}^{\sigma'}\mathcal{L}(c)^{\sigma}$$ is set of representatives for $\Gamma_{N^{\sigma'}}\backslash {}^{\sigma'}\Gamma^{\sigma}\slash \Gamma_{N^{\sigma}}.$
\end{prop}

\begin{proof}
See \cite{BMP1}.
\end{proof}

\begin{definition}\label{wert}
For $\gamma>0$ and $a=(a_1,a_2)\in \R^2$ with $a_j>6$ define $\mathcal{K}_a(\gamma)$ as the set of functions $k=k_1\times k_2$ with $k_j$ an even holomorphic function on $|\Re \nu| \leq 2\gamma$ satisfying $$k_j(\nu) \ll e^{-|\Im \nu|}(1+|\Im \nu|)^{-a_j}.$$

\end{definition}

\begin{definition} 
Let $\delta$ be a character of $\Gamma.$ For $\sigma,\sigma' \in \mathcal{P},r \in t_{\sigma}', s \in t_{\sigma'}',$ we define the Kloosterman sum $$S^{\sigma,\sigma'}_{\delta}(r,s,c):=\sum_{\left( \begin{array}{cc}
a &b \\ c & d \\ \end{array} \right) \in 
g_{\sigma'}^{-1}{}^{\sigma'}\mathcal{L}(c)^{\sigma}g_{\sigma}} \left(\delta(g_{\sigma'} \begin{pmatrix}
a &b \\ c & d \\ \end{pmatrix} g_{\sigma}^{-1})\right)^{-1} e(-Tr(\frac{rd}{c}+\frac{sa}{c})).$$

\end{definition}

\begin{definition}
For $t \in \C^{*,2}, \nu \in \C^2$ let $$\mathbf{B}(t,\nu):=\prod_{j=1}^2 \frac{I_{\nu_j}(4\pi\sqrt{t_j})I_{\nu_j}(4\pi\overline{\sqrt{t_j}})-I_{-\nu_j}(4\pi\sqrt{t_j})I_{-\nu_j}(4\pi\overline{\sqrt{t_j}})}{\sin \pi \nu_j}.$$
For $k \in \mathcal{K}_a(\psi),$ we define the Bessel transform of $k$ on $\C^{*,2}$ by $$Bk(t):=\prod_{j=1}^2 \frac{i}{2}\int_{\Re(\nu_j)=0} k_j(\nu_j)\mathbf{B}(t_j,\nu_j)\sin(\pi \nu_j) d\nu_j.$$
\end{definition}

To prove Theorem \ref{ftt0}, we will always take the cusps $\sigma,\sigma'$ to be not equal so there will be no ``delta" term standardly seen in such Kuznetsov trace formulas. The most general form of the trace formula with possible equivalent cusps is in the appendix, Section \ref{sec:app}.

For economy, let $$CSC_{r,s}^{\sigma,\sigma'}(k):=\sum_{\zeta \in \mathcal{P}} \sum_{\mu} \int_{\infty}^\infty \overline{D_{r,\sigma}(P^{\zeta},\frac{iy}{2},i\mu)}D_{s,\sigma'}(P^{\zeta},\frac{iy}{2},i\mu)d_{r,\sigma}(\mu)\overline{d_{s,\sigma'}(\mu)}k(\frac{iy}{2}+i\mu)dy,$$ where $\zeta$ is a sum over the cusps and $\mu$ is a lattice depending on $P^{\zeta}.$ 

\begin{theorem}(Bruggeman-Miatello-Kuznetsov Trace Formula)\label{kuk}
Let $\sigma \neq \sigma' \in \mathcal{P},$ with character $\kappa \theta$ of $\Gamma=\Gamma_1(4) \cap \Gamma_0(D).$  For $k \in \mathcal{K}_a(\psi),$  we have
\begin{multline}\label{eq:bmtr078} \sum_{l \geq 1} \overline{c_{r,\sigma}(f_l)}c_{s,\sigma'}(f_l)d_{r,\sigma}(\mu_l)\overline{d_{s,\sigma'}(\mu_l)}k(\mu_l)+CSC_{r,s}^{\sigma,\sigma'}(k)=\\ (\frac{2}{\pi})^2   \sum_{c \in {}^{\sigma'}\mathcal{C}^{\sigma}} S^{\sigma,\sigma'}_{\kappa\theta}(r,s,c) |\N(c)|^{-1}  Bk(\frac{ rs}{c^2}).
\end{multline}

\end{theorem}

\section{The trace formula for explicit cusps}\label{sec:fig}

Let us take for our essential cusps $g_{\sigma'}\cdot \infty=D-1$ with $g_{\sigma'}=\left( \begin{array}{cc}
D-1 & l \\ 1 & m \\ \end{array} \right)$ and $D,l\equiv 1(4), m\equiv 0(4D),$ and $g_{\sigma} \cdot \infty$  with $g_{\sigma}=\left( \begin{array}{cc}
1 & 0 \\ 0 & 1 \\ \end{array} \right).$
We note they are not $\Gamma_0(D)$-equivalent. 

We now unravel the definitions of ${}^{\sigma'}\mathcal{C}^{\sigma}$ and ${}^{\sigma'}\mathcal{L}(c)^{\sigma}$ for our chosen cusps from the Kuznetsov trace formula above. Recall $C$ is the big Bruhat cell and have elements of the form $g=\left( \begin{array}{cc}
a & b \\ c & d \\ \end{array} \right)$ with $c\neq0,a,b,c,d \in \F.$ Assume $g \in g_{\sigma'}^{-1}{}^{\sigma'}\mathcal{L}(c)^{\sigma}g_{\sigma}=g_{\sigma'}^{-1}{}^{\sigma'}\mathcal{L}(c)^{\sigma} .$   We must have then $g_{\sigma'}g \in {}^{\sigma'}\Gamma^{\sigma} \subset \Gamma_0(D)\cap \Gamma_1(4),\mbox{ or }$  $$\left( \begin{array}{cc}
(D-1)a+lc & (D-1)b+ld \\ -a+mc & -b+md \\ \end{array} \right) \in {}^{\sigma'}\Gamma^{\sigma} \subset \Gamma_0(D) \cap \Gamma_1(4).$$
This implies $-a+mc \equiv 0(D)$ which implies $a\equiv 0(D),$ but as $ad-bc=1,$ it also follows $(c,D)=1.$ This implies $a,d\equiv 0(4),c\equiv 1(4),$ and $b\equiv -1(4).$

So it follows, using the extension of $\kappa$ to $\Gamma_2,$ that $$\kappa(g_{\sigma'} \left( \begin{array}{cc}
a & b \\ c & d \\ \end{array} \right))=\kappa(g_{\sigma'})\left(\frac{a}{c}\right)_4=\left(\frac{D}{1}\right)_4 \left(\frac{a}{c}\right)_4=\left(\frac{a}{c}\right)_4.$$

We also have $$\theta(g_{\sigma'} \left( \begin{array}{cc}
a & b \\ c & d \\ \end{array} \right))=\theta(-b+md)=\theta(-b).$$ With $ad-bc=1$ and $a\equiv 0(D),$ we get $-b\equiv \overline{c} (D),$ so $\theta(-b)=\overline{\theta}(c).$

Recall that ${}^{\sigma'}\mathcal{L}(c)^{\sigma}$ is a set of representatives for $\Gamma_{N^{\sigma'}}\backslash {}^{\sigma'}\Gamma(c)^{\sigma} \slash \Gamma_{N^{\sigma}}.$ Further, the map $$(\gamma_{\sigma'},\gamma_{\sigma}) \in \Gamma_{N^{\sigma'}} \times \Gamma_{N^{\sigma}} \to \gamma_{\sigma'}\gamma\gamma_{\sigma}\in \Gamma$$ is one-to-one and the function $\gamma \to c(\gamma)$ is constant on each double class in $\Gamma_{N^{\sigma'}} \backslash \Gamma \slash \Gamma_{N^{\sigma}}.$ Therefore, the elements in the Kloosterman sum, $$\left( \begin{array}{cc}
a & b \\ c & d \\ \end{array} \right) \in g_{\sigma'}^{-1}{}^{\sigma'}\mathcal{L}(c)^{\sigma}g_{\sigma}$$ are parametrized for a fixed $c$ by $d\mod(c).$ It is immediate then that we can rewrite the trace formula in a more explicit form now:

 \begin{theorem}(Bruggeman-Miatello-Kuznetsov Trace Formula (Variant))\label{calk}

Define cusps $g_{\sigma'}\cdot \infty=D-1$ and $g_{\sigma} \cdot \infty$ with $g_{\sigma'}=\left( \begin{array}{cc}
D-1 & l \\ 1 & m \\ \end{array} \right)$ with $D,l\equiv 1(4), m\equiv 0(4D)$ as well as $g_{\sigma}=\left( \begin{array}{cc}
1 & 0 \\ 0 & 1 \\ \end{array} \right).$ Let  $\kappa \overline{\theta}$ be the character of $\Gamma=\Gamma_1(4) \cap \Gamma_0(D).$ For $k \in \mathcal{K}_a(\psi),$ we have
\begin{multline}\label{eq:bmtr078} \sum_{l \geq 1} \overline{c_{r,\sigma}(f_l)}c_{s,\sigma'}(f_l)d_{r,\sigma}(\mu_l)\overline{d_{s,\sigma'}(\mu_l)}k(\mu_l)+CSC_{r,s}^{\sigma',\sigma}(k)=\\ (\frac{2}{\pi})^2 \sum_{\substack{c \in \ri \\ c\equiv 1(4)\\(c,D)=1}} \theta(c) S_{4}(r,s,c) |\N(c)|^{-1}  Bk(\frac{ rs}{c^2})
\end{multline}

with $$S_{4}(r,r',c):=\sum_{ad\equiv 1(c)}(\frac{a}{c})_4 e(Tr(\frac{r'd}{c}+\frac{ra}{c})).$$
\end{theorem}

\section{Finding the residual Eisenstein series}\label{sec:err}

The next theorem is inspired by Theorem $3.1$ from \cite{LP}.

\begin{theorem}\label{exi}
\begin{enumerate}
Let $D\in \ri $ with $D\equiv 1(4).$ 
\item Let  $L^2_d(\Gamma\backslash G\slash K,\overline{\theta} \kappa )$  denote the space of $f_l$ satisfying $$f_l(\gamma g)=\overline{\theta}(\gamma)\kappa(\gamma) f_l(g)$$ for a Dirichlet character $\theta \mod D,$ and $\gamma \in \Gamma_1(4) \cap \Gamma_0(D).$ 
Then there  exists a $f_l\in L^2_d(\Gamma\backslash G\slash K, \overline{\theta} \kappa )$ that is an eigenform for the Laplace operator $L=\{L_1,L_2\}$ with spectral parameter $\bigg(1-\mu_l^2,1-\mu_l^2\bigg)$ with $\mu_l=\frac{1}{4}$ only if $\overline{\theta}^4$ is trivial modulo $D.$ Label this form with $\overline{\theta}^4 \equiv 1 \mod (D),$ $f_{00}.$ 
\item Let $\alpha$ be the best bound toward the Ramanujan conjecture for automorphic representations over $\F$ (the current best bound is $\alpha=\frac{7}{64}).$
 Suppose that $f_l\neq f_{00}$ is an eigenform for the Laplace operator $L$ with spectral parameter $\bigg(1-\mu_{l,1}^2,1-\mu_{l,2}^2\bigg),$ with either $\mu_{l,j} \in (0,1],$  then we have for $j \in \{1,2\},$ $$\Re(\mu_{l,j}) \leq \frac{\alpha}{4}.$$

\end{enumerate}

\end{theorem}

\begin{proof}

Incorporating the above expressions for ${}^{\sigma'}\mathcal{C}^{\sigma}$ and ${}^{\sigma'}\mathcal{L}(c)^{\sigma}$ from Section \ref{sec:fig}, we have  \begin{multline}\label{eq:exes}\psi_{\sigma,\sigma'}(r,s)=\sum_{c \in \mathbf{e}_{\F,R}\backslash{}^{\sigma'}\mathcal{C}^{\sigma}} \frac{1}{\N(c)^{s}} \sum_{\left( \begin{array}{cc}
a &b \\ c & d \\ \end{array} \right) \in {}^{\sigma'}\mathcal{L}(c)^{\sigma}} \left(\overline{\theta}\kappa( \begin{array}{cc} a &b \\ c & d \\ \end{array} )\right)^{-1} e(-Tr(\frac{rd}{c}))=\\  \sum_{\substack{c \in \mathbf{e}_{\F,R}\backslash \ri \\ c\equiv 1(4)\\(c,D)=1}} \frac{1}{\N(c)^{s}}\theta(c) \sum_{a(c)^{*}}\overline{\left(\frac{a}{c}\right)_4} e(-Tr(\frac{r\overline{a}}{c})). \end{multline}

Note the eigenvalue of the Laplace operator of $E(P^{\sigma},s,0,n\left( \begin{array}{cc}
\sqrt{z} & 0\\ 0 & \frac{1}{\sqrt{z}} \\ \end{array} \right)k)$ is $\bigg( (2-s)s, (2-s)s\bigg).$ If we normalize $s=1+i\mu,$ then the eigenvalue is $\bigg(1-\mu^2,1-\mu^2\bigg)$ as in the statement of the proposition. So the goal is to show that the Eisenstein series has a pole in its constant term at $s=\frac{5}{4},$ if $\overline{\theta}^4$ is trivial. This implies the existence of a Residual Eisenstein series. This Residual Eisenstein series would have $\mu=\frac{1}{4}$ completing the proposition.


For $r=0,$ by an orthogonality of characters argument, $\psi_{\sigma,\sigma'}(0,s)$ is non-zero precisely when $c=c_0^4,$ for $c_0 \in \ri.$ This gives $$\psi_{\sigma,\sigma'}(0,s)=\sum_{\substack{c_0 \in \mathbf{e}_{\F,R}\backslash \ri \\ c_0\equiv 1(4)}} \frac{\theta^4(c_0)\phi(c_0^4) }{\N(c_0)^{4s}}.$$

We use characters modulo $4$ and that $\phi(c_0^4)=\N(c_0)^3\phi(c_0)$ to write the sum as $$\psi_{\sigma,\sigma'}(0,s)=\frac{1}{\phi(4)} \sum_{\lambda(4)} \sum_{\substack{c_0 \in \mathbf{e}_{\F,R}\backslash \ri }}\frac{\phi(c_0)\theta^4(c_0)\lambda(c_0)}{\N(c_0)^{4s-3}}.$$ Now take the term with trivial $\lambda=1,$ the sum in $c_0$ is invariant under transformation by a unit so can be written in terms of integral ideals  $$\frac{1}{\phi(4)} \sum_{\epsilon \in \mathbf{e}_{\F,R}\backslash \ri^{*}} \sum_{(c_0) }\frac{\phi(c_0)\theta^4(c_0)}{\N(c_0)^{4s-3}}=\frac{1}{\phi(4)}\sum_{\epsilon \in \mathbf{e}_{\F,R}\backslash \ri^{*}} \frac{L(4s-4,\theta^4)}{L(4s-3,\theta^4)}.$$

The outside sum over the units is finite as $\mathbf{e}_{\F,R}$ is of finite index in $\ri^{*},$ while the inside sum is a ratio of Hecke L-functions which only has a simple pole at $s=\frac{5}{4}$ when $\theta^4$ is trivial. For nontrivial $\lambda$ the same argument gives a Dirichlet series with analytic continuation to the complex plane. Otherwise, there are no poles for the series and therefore no residual Eisenstein series. 

\bigskip
 \indent We now prove the second statement using the Shimura correspondence \cite{F} and the best bound towards the Ramanujan conjecture for a number field \cite{BB}. The latter being an offspring of the result over $\Q$ from \cite{KSa}. Specifically, the Shimura correspondence tells us that the metaplectic forms are in one-to-one correspondence with certain automorphic forms. A metaplectic form on the quartic cover of $GL_2$ with spectral parameter $\bigg(1-\mu_{l,1}^2,1-\mu_{l,2}^2\bigg)$ will correspond via the Shimura correspondence to an automorphic form of spectral parameter $\bigg(1-4\mu_{l,1}^2,1-4\mu_{l,2}^2\bigg).$ This comes directly from the Lemma of section $2.1$ in \cite{F}. The result of \cite{BB} gives that if there exists an exceptional spectral parameter $\mu_{l,1},\mu_{l,2}$ for an automorphic form, then it must satisfy $\Re(\mu_{l,j}) \leq \frac{7}{64}.$ So by the Shimura correspondence, the metaplectic form associated to this automorphic form with exceptional spectral parameter must then satisfy $4\Re (\mu_{l,j}) \leq \frac{7}{64},$ proving the second statement.

\end{proof}

We also now can state our main theorem explicitly defining $K$ from Theorem \ref{ftt0}.

\subsection{The explicit Theorem \ref{ftt0}}\label{sec:mint}

 Consider metaplectic forms in the space $L^2_d(\Gamma \backslash G\slash K, \theta \kappa).$ We see from Theorem \ref{exi} there exists a residual Eisenstein series for this space if $\overline{\theta}^4 \equiv 1 \mod (D).$ So we again denote our theta function as $f_{00}$ with corresponding Fourier coefficients as above.

Denote the Mellin transform over $\R$ of $\Psi$ as $\hat{\Psi}(s),$ specifically $$\hat{\Psi}(s)=\int_0^\infty \Psi(y)y^{s-1}dy.$$ 

        
        Then we have \begin{thm}\label{ftt}{\bf{ Theorem 1.3 (Explicit)}}
Let $\Psi \in C^\infty_0(\R^{+}).$ Suppose $A,B,F,D \in \ri=\Z[\omega_8]$ are such that  $D\equiv 1(4),$ that $B=4B'$ and that any prime $p$ dividing $B'$ also divides $D,$ that  $\frac{B^2}{16A} \in \ri.$ Suppose also that $A,B$ are squares. Let $\theta$ be a Dirichlet character modulo $D.$ Let $\sigma,\sigma'$ denote certain cusps associated to the discrete subgroup $\Gamma_0(D) \cap \Gamma_1(4).$ Let $\alpha$ be the best bound toward the Ramanujan conjecture for automorphic representations over $\F$ (the current best bound is $\alpha=\frac{7}{64}).$ Let $f_{00}$ be the residual Eisenstein series in $L^2_d(\Gamma\backslash G\slash K,\overline{\theta} \kappa ).$  Then for any $\epsilon > \frac{\alpha}{8}$ and recalling $K_{\frac{DB^2}{16A},\frac{B^2}{16A}}$ from Theorem \ref{crrm} we have

\begin{itemize}

\item If $\theta^4 \equiv 1(D),$
\begin{multline}\label{eq:fyi12} \sum_{\substack{c \in \ri\\ c\equiv 1(4)\\(c,D)=1}} \frac{\prod_{j=1}^2\Psi(\frac{\sqrt{X}}{|c^{\eta_j}|^2})}{\N(c)}\theta(c)\sum_{x(c)} e(Tr(\frac{Ax^4+Bx^2+F}{c}))=\\ X^{\frac{1}{2}} \frac{\pi^3}{2\phi(4)} \hat{\Psi}(-\frac{1}{2}) ^2 T_{A,D}(\theta) +\\  
\frac{X^{\frac{1}{4}} \hat{\Psi}(-\frac{1}{4})^2}{g(-\frac{1}{4})^2} (\frac{\pi}{2})^2 \overline{c_{\frac{B^2}{16A},\sigma}(f_{00})}c_{\frac{B^2}{16A},\sigma'}(f_{00})d_{\frac{B^2}{16A},\sigma}(\frac{1}{4})\overline{d_{\frac{B^2}{16A},\sigma'}(\frac{1}{4})} +\\
X^{\frac{1}{4}} \hat{\Psi}(-\frac{1}{4})^2   K_{\frac{DB^2}{16A},\frac{B^2}{16A}}+O(X^{\epsilon}),\end{multline} 

where \begin{equation}
 T_{A,D}(\theta)
 := \left\{ \begin{array}{ll} \N(A) & \text{if } \theta \equiv \mathbf{1}(D); \medskip \\
        0 & \text{if else},\end{array} \right. \end{equation} 

 and $$g(-\frac{1}{4})=[{(4\pi^2)^{1/4}\Gamma(3/4)^2 \N(\frac{B^2}{16A})^{1/8}}]^{-1}.$$

\item else if $\theta^4 \not\equiv 1(D),$ then 
\begin{equation*}\label{eq:fyo0}\sum_{\substack{c \in \ri\\ c\equiv 1(4)}} \frac{\prod_{j=1}^2\Psi(\frac{\sqrt{X}}{|c^{\eta_j}|^2})}{\N(c)}\theta(c)\sum_{x(c)} e(Tr(\frac{Ax^4+Bx^2+F}{c}))= O(X^{\epsilon}).
\end{equation*}

\end{itemize}

\end{thm}

\section{From Kloosterman sums to quartic exponential sums}\label{sec:qkq}


 If we choose $\sigma=\begin{pmatrix} 1 & 0\\ 0& 1 \end{pmatrix}$ then from Section \ref{sec:chcc}  the set $t_{\sigma}'$ equals $\{r \in \F: r \in \mathcal{D}_F^{-1}\}.$ An element of this set can be written as $r=\frac{s}{\delta}$ with $s \in \ri$ and $\delta$ a generator of the different ideal $\mathcal{D}_F.$  We use the shorthand $e(x):=\exp(2\pi i Tr (\frac{x}{\delta}))$ for an additive character in this section. 
 \begin{theorem}\label{c4}
Let $A,B, p \in \ri.$ For $\N(p) \equiv 1(4)$ and $(AB,p)=1$ we have
 \begin{equation} \label{eq:c41}
\sum_{x(p)}e(\frac{Ax^4+Bx^2}{p})-\sum_{x(p)}e(\frac{Ax^2+Bx}{p})=  \sum_{x(p)} (\frac{x}{p})_2 e(\frac{Ax^2+Bx}{p})\end{equation} with $$\sum_{x(p)} (\frac{x}{p})_2 e(\frac{Ax^2+Bx}{p})=(\frac{AB}{p})_2e(\frac{-B^2\overline{8 A}}{p}) S_4(\overline{4^2A}B^2,\overline{4^2A}B^2,p)
.$$
\end{theorem}

Before we proof the theorem we require a lemma on quadratic exponential sums.
\begin{lemma}\label{hfm}
Suppose $m,n,r,p \in \ri$ with  $(p,2)= (m,p)=1$ and $p$ prime. Then we have  
$$\sum_{x(p)} e(\frac{mx^2+nx+r}{p})= (\frac{m}{p})_2 e(\frac{r}{p}) e(\frac{-n^2\overline{4m}}{p}) \sqrt{\N(p)} .$$

\end{lemma}

\begin{proof}
 By completing the square and a change of variables for any integer $c, (c,2)=1$ we have $$\sum_{x(c)} e(\frac{mx^2+nx+r}{c})=e(\frac{r}{c}) e(\frac{-n^2\overline{4m}}{c}) \sum_{x(c)} e(\frac{mx^2}{c}).$$  
 
 The rest of the proof we study $$ \sum_{x(c)} e(\frac{mx^2}{c})$$ for $c=p$ prime. Following Theorem 155 of \cite{Hec}
we have $$\sum_{x(p)} e(\frac{mx^2}{p})=(\frac{m}{p})_2\sum_{x(p)} e(\frac{x^2}{p})$$ as $(m,p)=1.$

Using characters of order $2$ $$\sum_{x(p)} e(\frac{x^2}{p})=\sum_{x(p)} (\frac{x}{p})_2 e(\frac{x}{p})=:\tau_p.$$ By looking at $\tau_p^2$ and a standard change of variables in the Gauss sum we get $\tau_p=(\frac{-1}{p})_2 \N(p).$ As $-1=i^2$ the quadratic character equals one and we are done.

\end{proof}

\begin{proof}\{Theorem \ref{c4}\}
To simplify notation in the proof we write $\eta$ as the quadratic character and $\rho$ as the quartic residue character modulo $p.$ 

First we investigate the multiplicative Fourier transform of the exponential sum in question, 

\begin{equation}
\sum_{A(p)}\overline{\chi(A)}\left[ \sum_{x(p)}  \eta(x) e(\frac{Ax^2+Bx}{p})\right].
\end{equation}

Assume $\chi =\mathbf{1}(p),$ by a change of variables $A \to A \overline{x^2},$ the sum in $x$ is zero as $(B,p)=1.$  

Now assume $\chi \not\equiv \mathbf{1}.$ By a change of variables $A \to A \overline{x^2}, x \to \overline{B}x$ this equals $$\chi(\overline{B^2})\eta(B)\tau(\chi^2 \eta)\tau(\overline{\chi}).$$

So by Fourier inversion \begin{equation}\label{eq:vi}\sum_{x(p)}  \eta(x) e(\frac{Ax^2+Bx}{p})= \eta(B) \frac{1}{\phi(p)} \sum_{\chi(p)} \chi(\overline{B^2}) \tau(\chi^2 \eta)\tau(\overline{\chi})\chi(A). \end{equation} 
The Hasse-Davenport relation gives the identity  \begin{equation}\label{eq:hd}
\tau(\chi^n)=\frac{-\chi(n^n) \prod_{l(n)} \tau(\chi \gamma^{l})}{\prod_{l(n)} \tau(\gamma^l)}
\end{equation} for $\chi \mod(p),$ and where $\gamma$ is the $n$-th residue character.

Writing the term $\tau(\chi^2 \eta)=\tau((\chi \rho)^2)=\tau((\chi \overline{\rho})^2),$ and using the Hasse-Davenport relation we can write 

$$\tau((\chi \rho)^2)=\frac{\chi(4) \tau(\chi \overline{\rho})\tau(\chi \rho)\tau(\eta)\eta(-1)}{\N(p)}.$$ This equality can be reached by using the equalities $\eta \rho=\overline{\rho}$ and $\overline{\tau(\eta)}=\eta(-1)\tau(\eta).$

So \eqref{eq:vi} equals $$ \frac{\eta(B)\tau(\eta)\eta(-1)}{\N(p)\phi(p)} \sum_{\chi(p)} \chi(\overline{B^2}) \chi(A) \left[\chi(4) \tau(\chi \overline{\rho})\tau(\chi \rho)\tau(\overline{\chi})\right].$$

Opening all the Gauss sums and rearranging sums this equals $$\frac{\eta(B)\tau(\eta)\eta(-1)}{\N(p)\phi(p)} \sum_{a(p)} \overline{\rho(a)}e(\frac{a}{p}) \sum_{b(p)} \rho(b)e(\frac{b}{p}) \sum_{c(p)} e(\frac{c}{p})\sum_{\chi(p)} \chi(4ab\overline{cB^2}A).$$

Using orthogonality of characters this equals 
$$\frac{\eta(B)\tau(\eta)\eta(-1)}{\N(p)} \sum_{\substack{a,b,c(p)\\ 4abA\equiv B^2c(p)}} \overline{\rho(a)}e(\frac{a}{p})  \rho(b)e(\frac{b}{p})  e(\frac{c}{p}).$$

We can rewrite this as $$\frac{\eta(B)\tau(\eta)\eta(-1)}{\N(p)} \sum_{a,b(p)} \rho(a) \overline{\rho(b)} e(\frac{a+b+4ab\overline{B^2}A}{p}).$$

A change of variables $b \to ba$ gives \begin{equation}\label{eq:cf}\frac{\eta(B)\tau(\eta)\eta(-1)}{\N(p)} \sum_{b(p)}\overline{\rho(b)}\sum_{a(p)} e(\frac{4bA\overline{B^2}a^2+(b+1)a}{p}).\end{equation}

 

 Using Lemma \ref{hfm} for \eqref{eq:cf} we get \begin{multline}\frac{\epsilon_p\sqrt{\N(p)}\eta(AB)\tau(\eta)\eta(-1)}{\N(p)} \sum_{b(p)}\eta(b)\overline{\rho}(b)e(\frac{-\overline{4^2 b A}B^2(b+1)^2}{p})=\\ \frac{\epsilon_p\sqrt{\N(p)}\eta(AB)\tau(\eta)\eta(-1)}{\N(p)} e(\frac{-B^2\overline{8 A}}{p})\sum_{b(p)} \rho(b) e(\frac{\overline{4^2A}B^2(b+\overline{b})}{p}).\end{multline}


As $\tau(\eta)=\epsilon_p \sqrt{\N(p)}$ and $\epsilon_p^2 \eta(-1)=1,$ we finally have $$\sum_{x(p)}  \eta(x) e(\frac{Ax^2+Bx}{p})=
\eta(AB) e(\frac{-B^2\overline{8 A}}{p})\sum_{b(p)} \rho(b) e(\frac{\overline{4^2A}B^2(b+\overline{b})}{p}).$$
\end{proof}

\subsection{Prime powers}\label{ppoww}


Recalling the notation of $$S_{4}(a,b,c)=\sum_{x(c)^{*}} (\frac{x}{c})_4 e(\frac{ax+b\overline{x}}{c})$$ from Theorem \ref{calk} we have the following proposition and lemma.

\begin{prop}\label{klo}
Let $a,b,p \in \ri.$ For $m>1,m \in \Z$  and $\N(p)\equiv 1(4)$ prime, \begin{equation}S_{4}(a,b,p^{2m})=\N(p)^{m}\sum_{\substack{u(p^{2m})\\u^2 \equiv b \overline{a}(p^{2m})}} (\frac{u}{p^{2m}})_4 e(\frac{2au}{p^{2m}}).
\end{equation}

\begin{equation}
S_{4}(a,b,p^{2m+1})= \N(p)^{m+1/2} (\frac{a}{p})_2\sum_{\substack{u(p^{2m+1})\\u^2 \equiv b \overline{a}(p^{2m+1})}} (\frac{u}{p^{2m+1}})_4 e(\frac{2au}{p^{2m+1}})
\end{equation}
 \end{prop}
 
 \begin{proof}
 See \cite{P1}.
 \end{proof}

\begin{lemma}\label{hfm1}
Suppose $m,n,r,c \in \ri$ with  $(c,2)= (m,c)=1$ and $c\equiv 1(4)$ . Then we have  

$$\sum_{x(c)} e(\frac{mx^2+nx+r}{c})= (\frac{m}{c})_2 e(\frac{r}{c}) e(\frac{-n^2\overline{4m}}{c}) \sqrt{\N(c)} .$$

\end{lemma}

\begin{proof}
Recall from Lemma \ref{hfm} that we can reduce to the case of studying $$ (\frac{m}{c})_2 \sum_{x(c)} e(\frac{x^2}{c}).$$ Theorem $163$ of \cite{Hec} implies $$ \sum_{x(c)} e(\frac{x^2}{c})=\sqrt{\N(c)} \frac{1}{\N(8)} \sum_{x(4)} e(\frac{-cx^2}{4}).$$ As $c \equiv 1 (4)$ the last sum equals $$\sum_{x(4)} e(\frac{-x^2}{4})$$ which equals $\sqrt{\N(8)}$ by the equation below $(202)$ in \cite{Hec}.

\end{proof}

 \begin{prop}\label{pow}
 Theorem \ref{c4} is true for prime powers if  $B=4B'$ with $A,B$ are squares. Namely, for $m>1$ \begin{equation} \label{eq:c41}
\sum_{x(p^m)}e(\frac{Ax^4+Bx^2}{p^m})-\sum_{x(p^m)} e(\frac{Ax^2+Bx}{p^m})= (\frac{AB}{p^m})_2 e(\frac{-B^2\overline{8 A}}{p^m})S_4(\overline{4^2A}B^2,\overline{4^2A}B^2,p^m).\end{equation} 
 \end{prop}
\begin{proof}


On the right hand side of \eqref{eq:c41} using Proposition \ref{klo} for odd powers, we have for $m \geq 1,$  \begin{multline}
\label{eq:okl}
(\frac{AB}{p^{2m+1}})_2 e(\frac{-B^2\overline{8 A}}{p^{2m+1}})\sum_{b(p^{2m+1})} (\frac{b}{p^{2m+1}})_4 e(\frac{\overline{4^2A}B^2(b+\overline{b})}{p^{2m+1}})=\\ (\frac{AB}{p^{2m+1}})_2 e(\frac{-B^2\overline{8 A}}{p^{2m+1}})\left[\N(p)^{m+1/2} (\frac{A}{p^{2m+1}})_2(e(\frac{B^2\overline{8A}}{p^{2m+1}})+e(\frac{-B^2\overline{8A}}{p^{2m+1}}))\right]=\\ (\frac{B}{p^{2m+1}})_2 \N(p)^{m+1/2}\left[1+e(\frac{-B^2\overline{4A}}{p^{2m+1}})\right].
\end{multline}

While for the even powers, the analogous calculation gives $$(\frac{B}{p^{2m}})_2 \N(p)^{m}\left[1+e(\frac{-B^2\overline{4A}}{p^{2m+1}})\right]=\N(p)^{m}\left[1+e(\frac{-B^2\overline{4A}}{p^{2m+1}})\right].$$

Now we can write for any $m>1,$ \begin{multline} \label{eq:podw} \sum_{\substack{x(p^{m})\\(x,p)>1}}e(\frac{Ax^4+Bx^2}{p^{m}})+\sum_{x(p^{m})} (\frac{x}{p^{m}})_2 e(\frac{Ax^2+Bx}{p^{m}})=\\ \sum_{x(p^{m})}  e(\frac{Ax^4+Bx^2}{p^{m}})-\sum_{x(p^{m})}  \mathbf{1}(x)e(\frac{Ax^2+Bx}{p^{m}}).\end{multline}

We use Proposition 5.2 from \cite{LP}. This proposition is an application of stationary phase. We will use the proposition on the first term on the above right hand side equation, the second term on the right hand side will follow by using Lemma \ref{hfm1}. Following their notation, we have $f(x)=Ax^4+Bx^2, f'(x)=4Ax^3+2Bx,$ and so the roots are $ \alpha=0$ and $\alpha^2=\overline{2A}B \pmod {p^m}.$ As $A,B$ are squares, such non-zero $\alpha$ exist as $\Q(\omega_8)$ contains $\sqrt{2}.$  

 We look at the odd powers first. It is easy to check that $f''(\frac{-B}{2A})=-4B,$ so $$\sum_{x(p^{2m+1})}  e(\frac{Ax^4+Bx^2}{p^{2m+1}})=\N(p)^{m+1/2}(\frac{B}{p^{2m+1}})_2\left[1+2e(\frac{-B^2\overline{4A}}{p^{2m+1}})\right].$$ 
Similarly, $$\sum_{x(p^{2m+1})}  e(\frac{Ax^2+Bx}{p^{2m+1}})=\N(p)^{m+1/2}(\frac{A}{p^{2m+1}})_2e(\frac{-B^2\overline{4A}}{p^{2m+1}})$$ by using Lemma \ref{hfm1} .

That the right hand side of \eqref{eq:podw} is actually $$\sum_{x(p^{2m+1})}e(\frac{Ax^4+Bx^2}{p^{2m+1}})-\sum_{x(p^{2m+1})} e(\frac{Ax^2+Bx}{p^{2m+1}}),$$ comes from the fact that $$\sum_{\substack{x(p^{m})\\(x,p)>1}}e(\frac{Ax^2+Bx}{p^m})=0$$ for $m$ odd or even.

The even powers are completely analogous. In this case using again Proposition 5.2 from \cite{LP}, $$\sum_{x(p^{2m})}  e(\frac{Ax^4+Bx^2}{p^{2m}})=\N(p)^{m}\left[1+2e(\frac{-B^2\overline{4A}}{p^{2m+1}})\right]$$ using the same argument as for the odd powers, and $$\sum_{x(p^{2m})}  e(\frac{Ax^2+Bx}{p^{2m}})=\N(p)^{m}e(\frac{-B^2\overline{4A}}{p^{2m}})$$

Again as $A,B$ are perfect squares, the left hand side of \eqref{eq:okl} equals the left hand side of \eqref{eq:podw}.

\end{proof}

 \begin{remark}
  Note the identity in Proposition \ref{klo} is the only point in the paper that requires the eighth roots of unity rather than the fourth roots of unity. In the case of Livn\'{e} and Patterson \cite{LP} their polynomial $f(x)=Ax^3+Bx\pmod p^{2m+1} $ has roots $0,\alpha^2= \overline{3A }B \pmod p^{2m+1}$ which exist as the cube roots of unity contains the ``square-root" of $3.$ So in their case the minimal field which is needed to study the trace formula for cubic metaplectic forms is the same as that for applying Katz's cubic identity. While in our case we need to make the field slightly larger than $\Q$ adjoined fourth roots of unity to apply the quartic identity above.
  \end{remark}


\subsection{Identity for composite numbers}\label{sec:coco}

 \begin{prop}\label{c400}
Let $A,B, c \in \ri.$ For $c \equiv 1(4),(AB,c)=1$ with $A,B$ squares we have 
 \begin{equation} \label{eq:c412}
\sum_{x(c)}e(\frac{Ax^4+Bx^2}{c})=(\frac{AB}{c})_2 e(\frac{-B^2\overline{8 A}}{c})S_4(\overline{4^2A}B^2,\overline{4^2A}B^2,c)+\sum_{x(c)}e(\frac{Ax^2+Bx}{c})+\{ \mbox{Cross terms}\}. 
 \end{equation} 
Here the  ``cross terms" are composed of the following terms: for any decomposition $c=nm$ with $(n,m)=1,n\neq c,m\neq c,$  $$\left[\sum_{x(n)}e(\frac{\overline{m}(Ax^2+Bx)}{n})\right]\left[(\frac{AB}{m})_2e(\frac{-B^2\overline{8 An}}{m})S_4(B^2\overline{n4^2A},B^2\overline{n4^2A},m)\right].$$

\end{prop}

\begin{proof}
It suffices to do this for $c=mn, (m,n)=1$ with $m=p^k$ and $n=q^j,$ $p,q$ primes in $\ri.$ Induction on the number of primes in the decomposition will complete the proposition.

Note by Chinese remainder theorem \begin{equation}\label{eq:suck} \sum_{x(c)}e(\frac{Ax^4+Bx^2}{c})=\left[\sum_{x(n)}e(\frac{\overline{m}(Ax^4+Bx^2)}{n})\right]\left[\sum_{x(m)}e(\frac{\overline{n}(Ax^4+Bx^2)}{m})\right].\end{equation}


We can write the right hand side of \eqref{eq:suck} using Theorem \ref{c4} and Proposition \ref{pow} as \begin{multline}\label{eq:s1s}
\left[\sum_{x(n)}e(\frac{\overline{m}(Ax^2+Bx)}{n})+(\frac{AB}{n})_2 e(\frac{B^2\overline{8 Am}}{n})S_4(-B^2\overline{m8A},n)\right]\times \\ \left[\sum_{x(m)}e(\frac{\overline{n}(Ax^2+Bx)}{m})+(\frac{AB}{m})_2 e(\frac{-B^2\overline{8 An}}{m})S_4(B^2\overline{n8A},m)\right].
\end{multline}

If we multiply out \eqref{eq:s1s}, the ``end terms" are natural and can be written as $$\sum_{x(nm)}e(\frac{Ax^2+Bx}{mn})=\sum_{x(c)}e(\frac{Ax^2+Bx}{c})$$ plus $$(\frac{AB}{c})_2 e(\frac{-B^2\overline{8 A}}{nm})S_4(-B^2\overline{8A},nm)=(\frac{AB}{c})_2e(\frac{-B^2\overline{8 A}}{c})S_4(-B^2\overline{8A},c),$$ with the first equality coming directly from Chinese remainder theorem and the second coming from twisted multiplicativity of Kloosterman sums: $$\sum_{b(m)} (\frac{b}{m})_4 e(\frac{\overline{4^2An}B^2(b+\overline{b})}{m})\sum_{b(n)} (\frac{b}{n})_4 e(\frac{\overline{4^2Am}B^2(b+\overline{b})}{n})=\sum_{b(mn)} (\frac{b}{mn})_4 e(\frac{\overline{4^2A}B^2(b+\overline{b})}{mn}).$$

The ``cross" terms of multiplying out \eqref{eq:s1s} are: \begin{multline}\label{eq:csh}\left[\sum_{x(n)}e(\frac{\overline{m}(Ax^2+Bx)}{n})\right]\left[(\frac{AB}{m})_2 e(\frac{-B^2\overline{8 An}}{m})S_4(-B^2\overline{n4^2A},m)\right]+\\ \left[\sum_{x(m)}e(\frac{\overline{n}(Ax^2+Bx)}{m})\right]\left[(\frac{AB}{n})_2e(\frac{-B^2\overline{8 Am}}{n})S_4(-B^2\overline{m4^2A},n).\right]\end{multline}

The point now is if $c=\prod_{j=1}^N p_j^{k_j},$ by multiplicativity of the individual Kloosterman sums and quadratic exponential sums, we can always reduce to a product of a single Kloosterman sum term times a single quadratic exponential sum. So again considering the case $c=nm$ with $(n,m)=1$ is sufficient.




\end{proof}

\section{From the trace formula to asymptotics of sums of exponential sums}\label{sec:mww}



Let $\Psi \in C^{\infty}_0(\R^{+}),$ with transform $\widetilde{\Psi}_X(s)=\int_0^\infty \Psi(\frac{X}{t})t^{s-1}.$ Write $\widetilde{\Psi}(s)=\widetilde{\Psi}_1(s).$ Our goal in this section is to prove the following theorem.

\begin{theorem}\label{al}
For $X$ large and for any $0<\epsilon<1/2$ we have
\begin{multline*}\label{eq:fsf00}
\sum_{c \in {}^{\sigma'}\mathcal{C}^{\sigma}} \frac{S^{\sigma,\sigma'}_{\theta\kappa}(r,r',c)}{\N(c)} \prod_{j=1}^2\Psi(\frac{X}{|c^{\eta_j}|^2})=\\ (\frac{\pi}{2})^2 \sum_{\substack{l \geq 1\\ \mu_l \mbox{ excep.}}} \overline{c_{r,\sigma}(f_l)}c_{r',\sigma'}(f_l)d_{r,\sigma}(\mu_l)\overline{d_{r',\sigma'}(\mu_l)}\prod_{j=1}^2 \frac{ \widetilde{\Psi}(\mu_{l,j})}{g(-\mu_{l,j})}X^{\mu_{l,j}}+O(\log X+X^{\epsilon-1}).
\end{multline*} Here excep. denotes metaplectic forms with exceptional spectral parameter.

\end{theorem}

Let us setup how to prove Theorem \ref{al}. We use the trace formula from Section \ref{sec:app} which is also stated in Theorem \ref{kuk}. The point of this section is to use a clever choice of test function in this trace formula to get an asymptotic for a sum of Kloosterman sums. Thankfully, the hard work has been done for us in \cite{BM1}. There they got an asymptotic for a group of real rank one. In fact the exact same notation is used in a construction of a Kuznetsov formula for a product of real rank one groups in \cite{MW}. The point of this section is to apply the test function built in \cite{BM1} to a group of real rank $2.$  The estimates in this higher real rank case our identical to the real rank one case and we cite them when we can.

To exploit the test functions in \cite{BM1}, we write from our above trace formula the Bessel transform as 

\begin{multline}
Bk(\frac{ mn}{c})=\prod_{j=1}^2 \frac{i}{2} \int_{\Re(\nu)=0} k(\nu_j) \bigg(I_{-\nu_j}(4\pi (\sqrt{\frac{ mn}{c}})^{\sigma_j})I_{-\nu_j}(4\pi \overline{ (\sqrt{\frac{ mn}{c}})^{\sigma_j}})-\\ I_{\nu_j}(4\pi  (\sqrt{\frac{ mn}{c}})^{\sigma_j})I_{\nu_j}(4\pi \overline{ (\sqrt{\frac{ mn}{c}})^{\sigma_j}})\bigg)d\nu_j=\\ =\prod_{j=1}^2 \int_{\Re(\nu_j)=0} k(\nu) g(\nu_j) \tau(\chi_{m},\chi_{n}, \begin{pmatrix} \frac{1}{|c^{\eta_j}|} & 0 \\ 0 & |c^{\eta_j}| \end{pmatrix},\nu_j) d\nu_j
\end{multline}

with $$g(\nu)=\frac{4(4\pi^2|m_1m_2|)^{\nu}\nu}{\Gamma(\nu+1)^2},$$ and $$
\tau(\chi_{m},\chi_{n},\begin{pmatrix} \frac{1}{|c^{\eta}|} & 0 \\ 0 & |c^{\eta}| \end{pmatrix},\nu)=\sum_{j,k \geq 0} \frac{(4\pi^2 \overline{(\frac{mn}{c})^{\eta}})^j(4\pi^2 (\frac{mn}{c})^{\eta})^k}{j!k!\Gamma(\nu+1+j)\Gamma(\nu+1+j)}.$$ 

This is the same function $g$ that shows up in Theorem \ref{ftt}. Also, $\chi_m(\begin{pmatrix} 1 & x \\ 0 & 1 \end{pmatrix})=exp(2\pi iTr (mx)).$ The important point of writing the Bessel transform in this way is that we can use estimates from \cite{BM1}. We will describe these estimates in Theorem \ref{al}.

\begin{proof}\{Theorem $9.1$\}

  We define the test functions associated to our trace formula as $$h_{j,X}(\nu)=\left[\frac{\widetilde{\Psi_X}(-\nu)-\widetilde{\beta_X}(-\nu)}{g(\nu)}+\frac{\widetilde{\Psi_X}(\nu)-\widetilde{\beta_X}(\nu)}{g(-\nu)})\right], j \in \{1,2\}$$ with $\beta_X$ defined in (\cite{BM1},pg.114). We remind that with this definition $h_X(\nu)=h_{1,X}(\nu_1)h_{2,X}(\nu_2)$ is an even function in $\nu_j$ and holomorphic on the strip $|\Re \nu_j| \leq \sigma_j$ with $\sigma_j$ larger than one.
\begin{remark} In \cite{BM1} their test function for a real rank one group is holomorphic on the strip $|\Re \nu| \leq \sigma$ for $\sigma$ larger than $\rho$ which is half the sum of the positive roots. \end{remark}

 For $|\Im \nu_j| \geq 1, |\Re \nu_j| \leq \sigma_j,$ we have $$h_{j,X}(\nu_j)=O(X^{|\Re \nu_j|} e^{-|\Im \nu_j|}(1+|\Im \nu_j|)^{-l}),$$ for any $l>0$ that is sufficiently large. We also have the estimate for $|\Im \nu_j| \leq 1,$ but $\pm \nu_j$ staying away from the zeroes of $g,$ $$h_{j,X}(\nu_j)=O(X^{|\Re \nu_j|}).$$ Near the zeroes of $g$ we have the estimate $$h_{j,X}(\nu_j)=O(X^{|\Re \nu_j|}(1+|\log X|)).$$ Since the zero of $g$ is at $\nu_j=-1,$ the size of this estimate is negligible. In particular $h_{X} \in K_{l}(\sigma)$(Definition \ref{wert}) for $l$ large and $\sigma_j>2$ and can be used in the trace formula.



 If $\nu_j \in (0,1]$ and is not a zero of $g$ (which will be the case in this paper), the automorphic form
associated to this spectral parameter is in the complementary series, and \begin{equation}\label{eq:comps} h_{X}(\nu)= \prod_{j=1}^2 \left[\frac{ \widetilde{\Psi}(\nu_j)}{g(-\nu_j)}X^{\nu_j}+O(X^{-\nu_j})\right]= \prod_{j=1}^2 \frac{ \widetilde{\Psi}(\nu_j)}{g(-\nu_j)}X^{\nu_j}+O(X^{-\min(\nu_1,\nu_2)}).\end{equation}

\begin{remark}
In the paper \cite{BM1}, the variable $X \to 0$ while in our case we have $X \to \infty.$ This is not a problem due to our definition of $h_X(\nu).$ Specifically, if $X \to 0$ then for $\nu \in \R,$ $$\frac{\widetilde{\Psi}(-\nu)X^{- \nu}}{g(\nu)}$$ will be the main term. This is their estimate  while if $X \to \infty$ then the main term is $$\frac{\widetilde{\Psi}(\nu)X^{ \nu}}{g(-\nu)}.$$ 


Thus inverting the estimates of \cite{BM1} so that $X \to \infty$ is straightforward.

\end{remark}

We now state a crucial estimate [\cite{BM1}, (8)]  $$\tau(\chi_{m},\chi_{n}, \begin{pmatrix} \frac{1}{|c|} & 0 \\ 0 & |c| \end{pmatrix},\nu) = 1 + O(\frac{|c|^{-2}}{1+|\Im \nu|})$$ on $-\epsilon \leq \Re \nu \leq \sigma$ for each $\epsilon \in (0,1/2)$ with the implicit constant depending on $\epsilon$ and $\sigma.$
 
 
 

We note $$\frac{Bh_X(\frac{rr'}{c^2})}{\N(c)}=\prod_{j=1}^2 \int_{\Re(\nu_j)=0} h_X(\nu_j) g(\nu_j) \tau(\chi_{r},\chi_{r'}, \begin{pmatrix} \frac{1}{|c^{\eta_j}|} & 0 \\ 0 & |c^{\eta_j}| \end{pmatrix},\nu_j) d\nu_j$$ 

is equal to $$\prod_{j=1}^2 \widetilde{h_{j,X}}((ma)^{\sigma_j})=\prod_{j=1}^2  \widetilde{h_{j,X}}\begin{pmatrix} \frac{1}{|c^{\eta_j}|} & 0 \\ 0 & |c^{\eta_j}|\end{pmatrix}$$ with analogous notation to \cite{BM1},\cite{MW}.
  
  Directly following the analysis of [\cite{BM1},3.3] we have \begin{equation}\label{eq:ada}\prod_{j=1}^2  \widetilde{h_{j,X}}\begin{pmatrix} \frac{1}{|c^{\eta_j}|} & 0 \\ 0 & |c^{\eta_j}|\end{pmatrix}=   \prod_{j=1}^2 \left[ \frac{1}{|c^{\eta_j}|^2}\Psi(\frac{X}{|c^{\eta_j}|^2})+\frac{1}{(|c^{\eta_j}|^2)}O(X^{\epsilon-1}+X^{-1})\right]\end{equation} for any $\epsilon>0.$ Using the fact that $\Psi$ is compactly supported and therefore $|c^{\eta_j}|^2 \sim X$, \eqref{eq:ada} can be written as \begin{equation}\label{eq:ada1}  \frac{1}{\N(c)}\prod_{j=1}^2\Psi(\frac{X}{|c^{\eta_j}|^2})+O(X^{\epsilon-2}).\end{equation}


Define $$F_{\Psi}(X):= \sum_{c \in {}^{\sigma'}\mathcal{C}^{\sigma}} \frac{S^{\sigma,\sigma'}_{\theta\kappa}(r,r',c)}{\N(c)} \prod_{j=1}^2\Psi(\frac{X}{|c^{\eta_j}|^2}).$$ 



Combining the Bruggeman-Miatello-Kuznetsov trace formula (Theorem \ref{kuk}) and following the estimates $(18)$ and $(19)$ of \cite{BM1} we have \begin{multline}\label{eq:difft}\sum_{l \geq 1} \overline{c_{r,\sigma}(f_l)}c_{r',\sigma'}(f_l)d_{r,\sigma}(\mu_l)\overline{d_{r',\sigma'}(\mu_l)}h_X(\mu_l)+CSC_{r,r'}^{\sigma,\sigma'}(h_X)-\\ vol(\Gamma_{N^{\sigma}}\backslash N^{\sigma})\alpha(\sigma,r,\sigma',r') \prod_{j=1}^2 \frac{i}{2}\int_{\Re(\nu_j)=0} h_{j,X}(\nu_j)\sin \pi \nu_j d\nu_j=\\ (\frac{2}{\pi})^2 F_{\Psi}(X) + O(X^{\epsilon-1}),
\end{multline}
where the bound $O(X^{-1/2})$ (from \cite{BM1}) is taken into the bound $O(X^{\epsilon-1})$ since $\epsilon>0$ can be chosen much smaller than $1/2.$
As we will always take nonequivalent cusps, the delta term is zero. For equivalent cusps, the estimate of this term is in \cite{BM1}.

Likewise, we break the spectral terms into those with exceptional spectral parameter and those that are not, which we label ``non-excep". The latter is estimated in [\cite{BM1},(20)]$$ \sum_{\substack{l \geq 1\\ \mu_l \mbox{ non-excep.}}}\overline{c_{r,\sigma}(f_l)}c_{r',\sigma'}(f_l)d_{r,\sigma}(\mu_l)\overline{d_{r',\sigma'}(\mu_l)}h_X(\mu_l)+CSC_{r,r'}(h_X) \ll O(1+\log X).$$
 
The other spectral terms, are only finite in number and using Theorem $1$ from \cite{BM1} we have \begin{multline}\label{eq:compe} \sum_{\substack{l \geq 1\\ \mu_l \mbox{ excep.}}} \overline{c_{r,\sigma}(f_l)}c_{r',\sigma'}(f_l)d_{r,\sigma}(\mu_l)\overline{d_{r',\sigma'}(\mu_l)}h_X(\mu_l)=\\\sum_{\substack{l \geq 1\\ \mu_l \mbox{ excep.}}} \overline{c_{r,\sigma}(f_l)}c_{r',\sigma'}(f_l)d_{r,\sigma}(\mu_l)\overline{d_{r',\sigma'}(\mu_l)}\prod_{j=1}^2 \frac{ \widetilde{\Psi}(\mu_{l,j})}{g(-\mu_{l,j})}X^{\mu_{l,j}}+O(1+X^{-1}).
\end{multline} 

Using \eqref{eq:difft} and \eqref{eq:compe} we get \begin{multline}\label{eq:fsf}
\sum_{c \in {}^{\sigma'}\mathcal{C}^{\sigma}} \frac{S^{\sigma,\sigma'}_{\theta\kappa}(r,r',c)}{\N(c)} \prod_{j=1}^2\Psi(\frac{X}{|c^{\eta_j}|^2})=\\ (\frac{\pi}{2})^2\sum_{\substack{l \geq 1\\ \mu_l \in \mbox{ excep.}}} \overline{c_{r,\sigma}(f_l)}c_{r',\sigma'}(f_l)d_{r,\sigma}(\mu_l)\overline{d_{r',\sigma'}(\mu_l)} \prod_{j=1}^2 \frac{ \widetilde{\Psi}(\mu_{l,j})}{g(-\mu_{l,j})}X^{\mu_{l,j}}+O(\log X+ X^{\epsilon-1})
\end{multline} follows for any $0<\epsilon<1/2.$ This concludes Theorem \ref{al}.
\end{proof}

\begin{remark}
One can compare the estimates here in this section to the totally real case in \cite{BMP1} where they also consider a smooth compactly supported test function $$\prod_{j=1}^d \Psi(\frac{X}{|c^{\eta_j}|})$$ where $d$ is the number of real embeddings of the given field.
\end{remark}

\section{Quadratic exponential sums}\label{sec:qdd}

Recall from Proposition \ref{c400} that if we want to study sums of quartic exponential sums then we must also understand quadratic exponential sums. In this section we find an asymptotic for sums of certain quadratic exponential sums: \begin{equation}\label{eq:qa}
\sum_{\substack{c \in \ri\\ c\equiv 1(4)\\(c,D)=1}} \frac{\prod_{j=1}^2\Psi(\frac{\sqrt{X}}{|c^{\eta_j}|^2}))}{\N(c)}\theta(c)\sum_{x(c)} e(Tr(\frac{Ax^2+Bx+F}{c}))
\end{equation} with $\Psi \in C^{\infty}_0(\R^{+}).$ For this section let us abbreviate $e(Tr(x))$ to $e(x).$

We prove

\begin{prop}\label{q22}
Let $A,B,F,D \in \ri.$ Assume $\Psi \in C^{\infty}_0(\R^{+})$ with $X>0$ large and $2 \nmid D, B=4B'$ and $A,B'\parallel D^{\infty}$ with $\frac{B^2}{16A} \in \ri$ and $\frac{B^2}{16A} \equiv \pm 1(4).$  Suppose $\theta$ is a Dirichlet character modulo $D.$ Then for any $M\geq 0,$ we have \begin{multline}\label{eq:qaf}\sum_{\substack{c \in \ri\\ c\equiv 1(4)}} \frac{\prod_{j=1}^2\Psi(\frac{\sqrt{X}}{|c^{\eta_j}|^2})}{\N(c)}\theta(c)\sum_{x(c)} e(\frac{Ax^2+Bx+F}{c})=\\ \frac{\sqrt{X}T_{A,D}(\theta) }{4} \int_{\C^2} \Psi(\frac{1}{|t|^2})e(\frac{B^2}{4AX^{1/4}t})e(\frac{F}{X^{1/4}t})\frac{dt}{|t|}+O(X^{-M})\end{multline}
with

\begin{equation}
 T_{A,D}(\theta)
 := \left\{ \begin{array}{ll} \N(A) & \text{if } \theta \equiv \mathbf{1}(D); \medskip \\
        0 & \text{if else}.\end{array} \right.  \end{equation}

\end{prop}


\begin{proof}

Using the above closed form expression for $\sum_{x(c)} e(\frac{Ax^2+Bx}{c})$ from Lemma \ref{hfm1} we have \begin{equation}\label{eq:qa1}
\sum_{\substack{c \in \ri\\ c\equiv 1(4)}} \frac{\prod_{j=1}^2\Psi(\frac{\sqrt{X}}{|c^{\eta_j}|^2})e(\frac{F}{c})}{
\sqrt{\N(c)}}\theta(c)(\frac{A}{c})_2 e(\frac{-\overline{4A}B^2}{c}).
\end{equation}
Applying elementary reciprocity which states $$e(\frac{\overline{A}}{B})=e(\frac{-\overline{B}}{A})e(\frac{1}{AB})$$ and quadratic reciprocity, \eqref{eq:qa1} equals 
 \begin{equation*}\label{eq:qa2}
\sum_{\substack{c \in \ri\\ c\equiv 1(4)}} \frac{\prod_{j=1}^2\Psi(\frac{\sqrt{X}}{|c^{\eta_j}|^2})e(\frac{B^2}{4Ac})e(\frac{F}{c})}{
\sqrt{\N(c)}}\theta(c)(\frac{c}{A})_2 e(\frac{\overline{c}B^2}{4A}).
\end{equation*}

The restriction on the $c$-sum we remove by Dirichlet characters modulo $4,$

 \begin{equation}\label{eq:qa3}
\frac{1}{\phi(4)}\sum_{\lambda(4)} \sum_{c \in \ri} \frac{\prod_{j=1}^2\Psi(\frac{\sqrt{X}}{|c^{\eta_j}|^2})e(\frac{B^2}{4Ac})e(\frac{F}{c})}{
\sqrt{\N(c)}}\lambda(c)\theta(c)(\frac{c}{A})_2 e(\frac{\overline{c}B^2}{4A}).
\end{equation}

We break the $c$-sum into arithmetic progressions modulo $Q:=LCM(4,A,D)$ to get 

 \begin{equation}\label{eq:qa4}
\frac{1}{\phi(4)}\sum_{\lambda(4)} \left[\sum_{d(Q)} \lambda(d)\theta(d)(\frac{d}{A})_2e(\frac{\overline{d}B^2}{4A})\right] \sum_{c\equiv d(Q)} \frac{\prod_{j=1}^2\Psi(\frac{\sqrt{X}}{|c^{\eta_j}|^2})e(\frac{B^2}{4Ac})e(\frac{F}{c})}{
\sqrt{\N(c)}} .
\end{equation}

 Define $\Psi(\frac{\sqrt{X}}{|c|^2}):=\prod_{j=1}^2\Psi(\frac{\sqrt{X}}{|c^{\eta_j}|^2}).$ Then using Poisson summation on the $k$-sum in $c=d+Qk,$ the interior sum can be written as $$\sum_{c\equiv d(Q)} \frac{\prod_{j=1}^2\Psi(\frac{\sqrt{X}}{|c^{\eta_j}|^2})e(\frac{B^2}{4Ac})}{\sqrt{\N(c)}}=\sum_{m \in \ri} e(\frac{-dm}{Q})\int_{\C^2} \Psi(\frac{\sqrt{X}}{|t|^2})e(\frac{B^2}{4At})e(\frac{F}{t})e(\frac{-mt}{Q})\frac{dt}{\sqrt{\N(t)}}.$$

In the integral, change variables $t \to X^{1/4}t$ to get \begin{equation}\label{eq:int1} \sqrt{X} \int_{\C^2}\Psi(\frac{1}{|t|^2})e(\frac{B^2}{4AX^{1/4}t})e(\frac{F}{X^{1/4}t})e(\frac{-mX^{1/4}t}{Q})\frac{dt}{\sqrt{\N(t)}}.\end{equation}

A standard integration by parts argument implies with $m\neq 0$ and large $X,$ \eqref{eq:int1} is $O(X^{-M})$ for any $M >0.$

So \eqref{eq:qa4} equals for any $M>0,$ \begin{equation}\label{eq:qa8}\frac{X}{\phi(4)}\sum_{\lambda(4)}  \left[\sum_{d(Q)} \lambda(d)\theta(d)(\frac{d}{A})_2e(\frac{\overline{d}B^2}{4A})\right]  \int_{\C^2} \Psi(\frac{1}{|t|^2})e(\frac{B^2}{4AX^{1/4}t})e(\frac{F}{X^{1/4}t})\frac{dt}{\sqrt{\N(t)}}+O(X^{-M}).\end{equation}

To make any further calculations we now use the following assumptions:  $B=4B'$ and $B'\parallel D^{\infty}$ with $\frac{B^2}{16A} \in \ri$ and $\frac{B^2}{16A} \equiv \pm 1(4).$ We make these assumptions as these are necessary conditions to study the sums of quartic exponential sums in this paper. Then the sum modulo $Q$ of \eqref{eq:qa8} equals $$\sum_{d(Q)} \lambda(d)\theta(d)(\frac{d}{A})_2e(\frac{\overline{d}B^2}{4A})=\sum_{d(Q)} \lambda(d)\theta(d)(\frac{d}{A})_2.$$ 


As we are assuming $(D,2)=1,$ by the Chinese remainder theorem this last sum equals $$\left[\sum_{d(LCM(A,D))}\theta(d)(\frac{d}{A})_2\right]\left[\sum_{a(4)}\lambda(a)\right].$$ So for a non-zero contribution we must have $\lambda\equiv 1(4).$ We state when the $d$-sum is non-zero as a proposition. Define $T_{A,D}(\theta):=\sum_{d(LCM(A,D))}\theta(d)(\frac{d}{A})_2.$ We reduce to the case that $D=p\in \ri,$ $p$ prime. So $A=p^j, j \geq 0.$ \begin{lemma}\label{ptp}
Let $T_{p^j,p}(\theta):=\sum_{d(p^j)}\theta(d)(\frac{d}{p^j})_2.$ For $A=p^j,$ \begin{equation}\label{eq:cca} T_{p^j,p}(\theta)=
\begin{cases}
\N(p)^j, & \mbox{ if } j>1,j\equiv 1(2) \mbox{ and } \theta=(\frac{\cdot}{p})_2,\\ \\
\N(p)^j, & \mbox{ if } j>1,j\equiv 0(2) \mbox{ and } \theta=1,\\ \\
\N(p), & \mbox{ if } j=1 \mbox{ and } \theta=(\frac{\cdot}{p})_2,\\ \\
\N(p), & \mbox{ if } j=0 \mbox{ and } \theta=1,\\ \\
0, & \mbox{ else.}

\end{cases}
\end{equation}

\end{lemma}
It is clear by multiplicativity, $$T_{A,D}(\theta)=\prod_{p^j || LCM(A,D)} T_{p^j,p}(\theta).$$ As we assume $A$ is a square, by the above Proposition, $T_{A,D}(\theta)=\N(A)$ if $\theta \equiv \mathbf{1}(D)$ and is zero if anything else.

Finally, we conclude for any $M \geq 0$ and $X$ large, \begin{multline}\label{eq:qaf}\sum_{\substack{c \in \ri\\ c\equiv 1(4)}} \frac{\prod_{j=1}^2\Psi(\frac{\sqrt{X}}{|c^{\eta_j}|^2})}{\N(c)}\theta(c)\sum_{x(c)} e(\frac{Ax^2+Bx+F}{c})=\\ \frac{\sqrt{X}T_{A,D}(\theta) }{\phi(4)}  \int_{\C^2} \Psi(\frac{1}{|t|^2})e(\frac{B^2}{4AX^{1/4}t})e(\frac{F}{X^{1/4}t})\frac{dt}{\sqrt{\N(t)}}+O(X^{-M})\end{multline}
with

\begin{equation}
 T_{A,D}(\theta)
 := \left\{ \begin{array}{ll} \N(A) & \text{if } \theta \equiv \mathbf{1}(D); \medskip \\
        0 & \text{if else}.\end{array} \right.  \end{equation}

\end{proof}

\section{Estimates of the ``Cross term"}\label{sec:tro}

Just as in the previous section to understand asymptotics of sums of quartic exponential sums, we need to study sums of the ``Cross terms" from \eqref{eq:csh}.   So we consider in this section \begin{equation}\label{eq:hard} \sum_{\substack{n,m \in \ri\\ (n,m)=1\\ nm\equiv 1(4)}} \frac{\theta_D(nm)\prod_{j=1}^2\Psi(\frac{\sqrt{X}}{|(nm)^{\eta_j}|^2})}{\N(nm)} \left[\sum_{x(m)}e(\frac{\overline{n}(Ax^2+Bx)}{m})\right]\left[e(\frac{-B^2\overline{8 Am}}{n})S_4(B^2\overline{m4^2A},B^2\overline{m4^2A},n)\right].\end{equation}  





\begin{theorem}\label{crrm}
 
 Let $\alpha$ be the best bound toward the Ramanujan conjecture for automorphic representations over $\F$ (the current best bound is $\alpha=\frac{7}{64}).$ Let $\epsilon > \frac{\alpha}{8}.$ For a constant $K_{\frac{DB^2}{16A},\frac{B^2}{16A}}$ depending on certain cusps $\sigma_a,\sigma_b$ and parameters $A,B,D$ defined below. We have \begin{gather*}\label{eq:harde} \sum_{\substack{n,m \in \ri\\ (n,m)=1\\ nm\equiv 1(4)}} \frac{\theta_D(nm)\prod_{j=1}^2\Psi(\frac{\sqrt{X}}{|(nm)^{\eta_j}|^2})}{\N(nm)} \left[\sum_{x(m)}e(\frac{\overline{n}(Ax^2+Bx)}{m})\right]\left[e(\frac{-B^2\overline{8 Am}}{n})S_4(-B^2\overline{m4^2A},n)\right]=\end{gather*}

$$ \left\{ \begin{array}{ll}   
X^{\frac{1}{4}}\hat{\Psi}(-\frac{1}{4})^2K_{\frac{DB^2}{16A},\frac{B^2}{16A}}+ O(X^{\epsilon}) & \text{if } \theta_D^4\equiv \mathbf{1}(D); \medskip \\
      O(X^{\epsilon}) & \text{if } \theta_D^4 \not\equiv \mathbf{1}(D).  \end{array} \right.   $$

The term $K_{\frac{DB^2}{16A},\frac{B^2}{16A}}$ is a convergent sum \begin{gather}\label{eq:wconr}  \frac{1}{g(\frac{-1}{4})^2} \sum_{\substack{m \in \ri\\ m\equiv 1(4)}} \frac{\theta_D( m)\N(m)^{1/4}}{\phi(m)} \sum_{\substack{\chi \mod (m)\\ \chi^4 \equiv 1\mod (m)}} \tau(\chi) \chi(8A(\overline{-B^2})) \times \\ \bigg[\overline{c_{\frac{4DB^2}{16A},\sigma_a}(f_{00})}c_{\frac{B^2}{16A},\sigma_b}(f_{00})d_{\frac{4DB^2}{16A},\sigma_a}(\frac{1}{4})\overline{d_{\frac{B^2}{16A},\sigma_b}(\frac{1}{4})}+\frac{\{\mbox{Remainder}\}}{X^{\frac{1}{4}} (\frac{\pi}{2})^2 \widetilde{\Psi}(\frac{1}{4})}\bigg]\end{gather} Here $f_{00,4Dm^2}$ is the quartic residual Eisenstein series of level $4Dm^2,$ and $c_{\frac{B^2}{16A},\sigma_a}(f_{00,Dm^2}), c_{\frac{B^2}{16A},\sigma_b}(f_{00,Dm^2})$ are its Fourier coefficients at cusps $\sigma_a, \sigma_b$ defined below. The $\{Remainder\}$ term is defined in \eqref{eq:ress}, it is the spectral sum of metaplectic forms excluding the residual Eisenstein series and is $o(X^{1/4}).$

\end{theorem} 

\begin{cor}\label{corw}
On the assumption we can interchange the $m$-sum and the residual Eisenstein series from the Remainder term, we have the simplification of Theorem \ref{crrm} for any $\epsilon> \frac{\alpha}{8},$ \begin{multline}\label{eq:wcnn}    K_{\frac{DB^2}{16A},\frac{B^2}{16A}}=\\ \frac{1}{g(\frac{-1}{4})^2}  (\frac{\pi}{2})^2\sum_{\substack{m \in \ri\\ m\equiv 1(4)}} \frac{\theta_D( m)\N(m)^{1/4}}{\phi(m)} \sum_{\substack{\chi \mod (m)\\ \chi^4 \equiv 1\mod (m)}} \tau(\chi) \chi(8A(\overline{-B^2})) \times \\ \bigg[\overline{c_{\frac{4DB^2}{16A},\sigma_a}(f_{00})}c_{\frac{B^2}{16A},\sigma_b}(f_{00})d_{\frac{4DB^2}{16A},\sigma_a}(\frac{1}{4})\overline{d_{\frac{B^2}{16A},\sigma_b}(\frac{1}{4})}\bigg]+O(1),\end{multline} where $\alpha$ is the best bound towards the Ramanujan conjecture for automorphic representations over $\F.$

\end{cor}

We give a short outline of the proof. The idea is for a fixed $m,$ to associate the $n$-sum of \eqref{eq:hard} to a spectral sum of metaplectic forms of level $4Dm^2.$ The first step is in reducing $\left[\sum_{x(m)}e(\frac{\overline{n}(Ax^2+Bx)}{m})\right]e(\frac{-B^2\overline{8 Am}}{n})$ to $\sqrt{\N(m)}e(\frac{-B^2\overline{8An}}{m}),$ an exponential sum with argument $m.$ This reduction uses elementary reciprocity and completing the square for a quadratic exponential sum. Second, as multiplicative characters are more naturally associated to sums of Kloosterman sums we write the exponential sum in terms of multiplicative characters, call such a multiplicative character $\Delta \mod (D).$ By a careful choice of cusps $\sigma_a, \sigma_b$ we can then relate the Kloosterman sum in consideration $\Delta(c) S_4(s\overline{4Dm^2},t,c)$ to a Kloosterman sum directly related to a spectral sum $S^{\sigma_a,\sigma_b}_{\kappa'\Delta}(s,t,cm),$ where $\kappa'$ is the inverse of the Kubota symbol defined in the proof. Applying the trace formula to get to the spectral side of a family of metaplectic forms we realize by using Theorem \ref{exi} that the asymptotic of the Kloosterman sums is controlled by the quartic residual Eisenstein series. It is of asymptotic size $(\frac{X}{\N(m)})^{1/4}.$ Incorporating back the $m$-sum we have our main term $ K_{\frac{DB^2}{16A},\frac{B^2}{16A}}.$ It converges as it is equal to the absolutely convergent sum \eqref{eq:hard}. 

Getting any information beyond the convergence of the entire sum, say convergence of the $m$-sum for just the residual Eisenstein series seems difficult. We need a better understanding of the Fourier coefficients of metaplectic forms dependence on the level. For example, Deshouillers and Iwaniec \cite{DeshI} conjecture certain bounds for Fourier coefficients of automorphic forms over $\Q$ in terms of the level. Namely, they conjecture for Fourier coefficients $c_{r,\sigma}(f)$ of an automorphic form $f$ with level $q,$ $$c_{r,\sigma}(f) \ll \\\frac{1}{\mu(q)^{1/2}},$$ where $\mu(q)=(w,\frac{q}{w})q^{-1}$ and $\sigma=\frac{u}{w},(u,w)=1.$ Assuming the analogous bound for metaplectic forms would give the absolute convergence of the $m$-sum for a fixed metaplectic form. So some deeper cancellation is occurring in this sum over the level of a spectral sum.

\begin{proof}

We first show that the sum   \begin{equation}\label{eq:dsfd} \sum_{\substack{n,m \in \ri\\ (n,m)=1\\ nm\equiv 1(4)}} \frac{\theta_D(nm)\prod_{j=1}^2\Psi(\frac{\sqrt{X}}{|(nm)^{\eta_j}|^2})}{\N(nm)} \left[\sum_{x(m)}e(\frac{\overline{n}(Ax^2+Bx)}{m})\right]\left[e(\frac{-B^2\overline{8 Am}}{n})S_4(-B^2\overline{m4^2A},n)\right]\end{equation} is equal to \begin{equation}\label{eq:dsfd0} \sum_{\substack{n,m \in \ri\\ (n,m)=1\\ n\equiv 1(4)\\m \equiv 1(4)}} \frac{\theta_D(nm)\prod_{j=1}^2\Psi(\frac{X}{|(nm)^{\eta_j}|^2})}{\N(nm)} \left[\sum_{x(m)}e(\frac{\overline{n}(Ax^2+Bx)}{m})\right]\left[e(\frac{-B^2\overline{8 Am}}{n})S_4(-B^2\overline{m4^2A},n)\right]\end{equation}

For a fixed $n,m$ with $nm \equiv 1(4)$ there exists a unique unit $\epsilon$ with global inverse $\overline{\epsilon}$ such that $\epsilon n \equiv 1 (4)$ and likewise $\overline{\epsilon} m \equiv 1(4).$ We can then write \eqref{eq:dsfd} as  \begin{equation}\label{eq:dsfd1} \sum_{\substack{n,m \in \ri\\ (n,m)=1\\ (\epsilon n)(\overline{\epsilon m}) \equiv 1(4)}} \frac{\theta_D(\epsilon n \overline{\epsilon}m)\prod_{j=1}^2\Psi(\frac{\sqrt{X}}{|(\epsilon n \overline{\epsilon} m)^{\sigma_j}|^2})}{\N(\epsilon n\overline{\epsilon}m)} \left[\sum_{x(m)}e(\frac{\overline{\epsilon n}(Ax^2+Bx)}{\overline{\epsilon} m})\right]\left[e(\frac{-B^2\overline{8 A\overline{\epsilon} m}}{\epsilon n})S_4(-B^2\overline{\overline{\epsilon}m4^2A},\epsilon n)\right].\end{equation} By relabeling the variables this equals \eqref{eq:dsfd0}.

By completing the square and using similar arguments to Lemma \ref{hfm} and Lemma \ref{hfm1} we have \begin{multline}\sum_{x(m)}e(\frac{\overline{n}(Ax^2+Bx)}{m})= (\frac{nA}{m})_2 e(\frac{-\overline{4nA}B^2}{m}) \sum_{x(m)}e(\frac{x^2}{m})= (\frac{n}{m})_2e(\frac{\overline{4mA}B^2}{n})e(\frac{-\overline{4A}B^2}{nm})\sum_{x(m)}e(\frac{x^2}{m})=\\ (\frac{n}{m})_2e(\frac{\overline{4mA}B^2}{n})e(\frac{-\overline{4A}B^2}{nm})\sqrt{\N(m)}.\end{multline} The last equality follows again from noting $m \equiv 1(4).$ Combining the exponentials outside the quadratic exponential sum with $e(\frac{-B^2\overline{8 Am}}{n})$ from the last bracket of \eqref{eq:hard} gives $$e(\frac{\overline{4mA}B^2}{n})e(\frac{-\overline{4A}B^2}{nm})e(\frac{-B^2\overline{8 Am}}{n})=e(\frac{B^2\overline{8 Am}}{n}) e(\frac{-\overline{4A}B^2}{nm})$$ using the fact that $(1-\overline{2})\equiv \overline{2}(n).$ Now by assumption $e(\frac{-\overline{4A}B^2}{nm})=e(\frac{\frac{-B^2}{4A}}{nm}),$ and following the argument at the beginning of Section\ref{sec:cs} $$e(\frac{\frac{-B^2}{4A}}{nm})=1+O(X^{-1/2}).$$ We ignore this term for now on (can also just relabel test function). 

By another application of elementary reciprocity, we can focus on the $n$-sum, \begin{equation}  \sum_{\substack{n \in \ri\\ (n,m)=1\\ n\equiv 1(4)\\m\equiv 1(4)}} \frac{\prod_{j=1}^2\Psi(\frac{\sqrt{X}}{|(nm)^{\eta_j}|^2}) \theta_D(n)(\frac{n}{m})_2}{\N(n)}e(\frac{-B^2\overline{8 An}}{m})S_4(-B^2\overline{m4^2A},n). \end{equation} Here we write $\chi_m$ for a multiplicative character $\chi$ modulo $m$ to show the dependence on the modulus. Then expanding an additive character as a sum of multiplicative characters  we have $$e(\frac{-B^2\overline{8 An}}{m})=\frac{1}{\phi(m)}\sum_{\chi_m} \tau(\chi_m) \overline{\chi_m}(-B^2\overline{8 An}).$$ By a change of variables we write $S_4(-B^2\overline{m4^2A},n)=(\frac{m}{n})_4 S_4(-B^2\overline{m^24^2A},-B^2\overline{4^2A},n).$


Let us consider the character $\overline{(\frac{n}{m})_4}\chi_m(n)$ as character modulo $m^2,$ writing it as $\omega_{m^2}(n):= \overline{(\frac{n}{m})_4}\chi_m(n).$ As $(4D,mn)=1$ we make another label for $\theta_D(n)\overline{(\frac{n}{m})_4}\chi_m =:  \omega_{4Dm^2}(n) \mod (4Dm^2).$

So we consider for $m \equiv 1(4)$ the sum  \begin{equation}  \sum_{\substack{n \in \ri\\ (n,m)=1,n\equiv 1(4)}} \frac{\prod_{j=1}^2\Psi(\frac{\sqrt{X}}{|(nm)^{\eta_j}|^2})\omega_{4Dm^2}(n)}{\N(n)}S_4(-B^2\overline{m^24^2A},-B^2\overline{4^2A},n). \end{equation}

\begin{lemma}\label{mea}
 Let $m \equiv 1(4),$ $\kappa'=\overline{\kappa}$ the complex conjugate of the quartic residue symbol. Define cusps $g_{\sigma_a} \cdot \infty=\begin{pmatrix}  f & l\\ 4Dm^2 & 1 \end{pmatrix} \cdot \infty$ and $ g_{\sigma_b} \cdot \infty= \begin{pmatrix} 0 & \frac{-1}{m} \\ m & 0 \end{pmatrix} \cdot \infty,$ where  $f\equiv 1(4), l\equiv 0(4)$ on $\Gamma_0(D) \cap \Gamma_1(4).$ Then the sum  \begin{gather*}  \sum_{\substack{n \in \ri\\ (n,m)=1,n\equiv 1(4)}} \frac{\prod_{j=1}^2\Psi(\frac{\sqrt{X}}{|(nm)^{\eta_j}|^2})\omega_{4Dm^2}(n)}{\N(n)}S_4(-B^2\overline{m^24^2A},-B^2\overline{4^2A},n)=\\ \N(m) \sum_{n \in {}^{\sigma_a}\mathcal{C}^{\sigma_b}} \frac{\prod_{j=1}^2\Psi(\frac{\sqrt{X}}{|(nm)^{\eta_j}|^2})}{\N(n)}S^{\sigma_a,\sigma_b}_{\kappa'\omega_{4Dm^2}}(4D \frac{B^2}{16A}, \frac{B^2}{16A},n) \end{gather*}
\end{lemma}

\begin{proof}

Given $\begin{pmatrix} x & y\\ w &z \end{pmatrix} \in \Gamma_0(Dm^2)\cap \Gamma_1(4)$ we have\begin{equation}\label{eq:ccal}g_{\sigma_a}^{-1} \begin{pmatrix} x & y\\ w &z \end{pmatrix} g_{\sigma_b}=\begin{pmatrix} am & \frac{b}{m}\\ cm &dm \end{pmatrix}\end{equation} with $a,b,c,d \in \ri \mbox{ and } d\equiv 0 \mod(4D), a \equiv 0(4), -b\equiv  c\equiv 1(4).$ It directly follows that this process can be reversed to get an element in $\Gamma_0(Dm^2)\cap \Gamma_1(4).$
Indeed, take an element $\begin{pmatrix} am & \frac{b}{m}\\ cm &dm \end{pmatrix}$ with $a,b,c,d \in \ri \mbox{ with } d\equiv 0 \mod(4D), a \equiv 0(4), -b \equiv c \equiv 1(4).$ Then 
\begin{multline}\begin{pmatrix}  f & l\\ 4Dm^2 & 1 \end{pmatrix}\begin{pmatrix} am & \frac{b}{m} \\ cm & dm \end{pmatrix}\begin{pmatrix} 0 & \frac{1}{m} \\ -m & 0 \end{pmatrix} =\begin{pmatrix}  f & l\\ 4Dm^2 & 1 \end{pmatrix}\begin{pmatrix} -b & a \\ -dm^2 & c \end{pmatrix} = \\ \begin{pmatrix}  -bf-ldm^2 &  f a+cl\\ -4bDm^2- dm^2 & 4aDm^2+c \end{pmatrix}.\end{multline} It easy to check this last matrix is in $\Gamma_0(Dm^2)\cap \Gamma_1(4).$

Recall the Kubota symbol is defined  by

\begin{equation*}
\kappa (\gamma) = \begin{cases} 
\left(\frac{a}{c}\right)_{\!4} & \ec{if $c\neq 0$}\\
1 &\ec{if $c=0$},
\end{cases} \qquad \ec{ where } \gamma=\begin{pmatrix}a&b\\ c&d\end{pmatrix}\in \Gamma_1(4),
\end{equation*} with $\left(\frac{a}{c}\right)_{\!4}$ the quartic power residue symbol.

Define \begin{equation*}
\kappa' (\gamma) = \begin{cases} 
\overline{\left(\frac{a}{c}\right)_{\!4}} & \ec{if $c\neq 0$}\\
1 &\ec{if $c=0$},
\end{cases} \qquad \ec{ where } \gamma=\begin{pmatrix}a&b\\ c&d\end{pmatrix}\in \Gamma_1(4),
\end{equation*} just the multiplicative inverse of $\kappa.$

 
 We can check that $$\begin{pmatrix} 1 & -l\\ -4Dm^2 & f \end{pmatrix}, \begin{pmatrix} -b & a \\ -dm^2 & c \end{pmatrix} \in \Gamma_0(Dm^2)\cap \Gamma_1(4),$$ 
  so \begin{multline}\label{eq:craz4} \kappa(\begin{pmatrix} 1 & -l\\ -4Dm^2 & f \end{pmatrix}\begin{pmatrix} -b & a \\ -dm^2 & c \end{pmatrix})=\kappa(\begin{pmatrix} 1 & -l\\ -4Dm^2 & f \end{pmatrix})\kappa(\begin{pmatrix} -b & a \\ -dm^2 & c \end{pmatrix}\begin{pmatrix} 0 & -1 \\ 1 & 0 \end{pmatrix}))=\\ (\frac{1}{-4Dm^2})_4 (\frac{a}{c})_4=(\frac{1}{D})_4 (\frac{1}{m})_2 (\frac{a}{c})_4=(\frac{a}{c})_4.\end{multline} The last equality following from $-4$ is a fourth power in $\ri,$ and the extension of $\kappa$ to $\Gamma_2.$


For a Dirichlet character we also see $$\Delta_{4Dm^2}(\begin{pmatrix}  -bf-ldm^2 &  f a+cl\\ -4bDm^2- dm^2 & 4aDm^2+c \end{pmatrix})=\Delta_{4Dm^2}( c).$$

Therefore, by the definition of a Kloosterman sum at cusps $g_{\sigma_a}, g_{\sigma_b},$ and a standard change of variables, we have $$S^{\sigma_a,\sigma_b}_{\kappa'\Delta}(s,t,cm)=\sum_{\substack{\left( \begin{array}{cc}
am &\frac{b}{m} \\ cm & dm \\ \end{array} \right) \in 
g_{\sigma_a}^{-1}{}^{\sigma_a}\mathcal{L}(cm)^{\sigma_b}g_{\sigma_b} \\ d \equiv 0(4D)\\m^2ad-bc=1}} \left(\kappa'\Delta_{4Dm^2}(g_{\sigma_{a}} \begin{pmatrix}
am & \frac{b}{m} \\ cm & dm \\ \end{pmatrix} g_{\sigma_{b}}^{-1})\right)^{-1} e(\frac{sdm}{cm}+\frac{tam}{cm})=$$

$$\Delta_{4Dm^2}(c)  \sum_{a,d \mod(c), ad\equiv 1(c)} (\frac{d}{c})_4 e(\frac{s\overline{4Dm^2}a+td}{c})=\Delta_{Dm^2}( c) S_4(s\overline{4Dm^2},t,c)$$ with $c \equiv 1(4).$ 

Let us now take $s= 4D \frac{B^2}{16A}$ and $t=\frac{B^2}{16A},$ and $\Delta_{4Dm^2}(x) = \omega_{4Dm^2}(x).$ To finally connect \eqref{eq:dsfd} to a geometric side of the trace formula in Theorem \ref{kuk}, we need to understand ${}^{\sigma_a} \mathcal{C}^{\sigma_b}.$ From \eqref{eq:ccal}, the set ${}^{\sigma_a} \mathcal{C}^{\sigma_b}=\{n\in \ri: n\neq 0, n \equiv 0 (m), n\equiv 1(4)\}.$ Hence, we can conclude for $m\equiv 1(4), $ \begin{multline}   \sum_{\substack{n \in \ri\\  n\equiv 1(4)}} 
\frac{\prod_{j=1}^2\Psi(\frac{\sqrt{X}}{|(nm)^{\eta_j}|^2})}{\N(n)}\omega_{4Dm^2}(n)S_4(-B^2\overline{m^24^2A},-
B^2\overline{4^2A},n)=\\ \N(m) \sum_{n \in {}^{\sigma_a}\mathcal{C}^{\sigma_b}} \frac{\prod_{j=1}^2\Psi(\frac{\sqrt{X}}{|(nm)^{\eta_j}|^2})}{\N(n)}S^{\sigma_a,\sigma_b}_{\kappa'\omega_{4Dm^2}}(4D \frac{B^2}{16A}, \frac{B^2}{16A},n). \end{multline}

\end{proof}

Following Theorem 7.1, there is only a residual Eisenstein series $f_{00}$ for the sum from Lemma \ref{mea} if $(\kappa'\omega_{4Dm^2})^4 \equiv 1 (4Dm^2),$ which implies $\chi_m^4 \equiv 1(m^2)$ and $\theta_D^4 \equiv 1 (D).$ Assuming this is the case  \begin{multline} \sum_{n \in {}^{\sigma_a}\mathcal{C}^{\sigma_b}} \frac{\prod_{j=1}^2\Psi(\frac{\sqrt{X}}{|(nm)^{\eta_j}|^2})}{\N(n)}S^{\sigma_a,\sigma_b}_{\kappa'\omega_{4Dm^2}}(4D \frac{B^2}{16A}, \frac{B^2}{16A},n)=\\ (\frac{\pi}{2})^2 \frac{X^{\frac{1}{4}} \widetilde{\Psi}(\frac{1}{4})^2}{\N(m)^{\frac{1}{4}}g(-\frac{1}{4})^2} \overline{c_{\frac{4DB^2}{16A},\sigma_a}(f_{00})}c_{\frac{B^2}{16A},\sigma_b}(f_{00})d_{\frac{DB^2}{16A},\sigma_a}(\frac{1}{4})\overline{d_{\frac{B^2}{16A},\sigma_b}(\frac{1}{4})}+\\  (\frac{\pi}{2})^2 \sum_{f_l \neq f_{00}}  \prod_{j=1}^2 \frac{ \widetilde{\Psi}(\mu_{l,j})}{g(-\mu_{l,j})^2}X^{\frac{\mu_{l,j}}{2}}\overline{c_{\frac{4DB^2}{16A},\sigma_a}(f_{l})}c_{\frac{B^2}{16A},\sigma_b}(f_{l})d_{\frac{4DB^2}{16A},\sigma_a}(\mu_l)\overline{d_{\frac{B^2}{16A},\sigma_b}(\mu_l)}+CSC_{\frac{4DB^2}{16A},\frac{B^2}{16A}}^{\sigma_a,\sigma_b}. \end{multline} 

Otherwise, if either $\chi_m^4 \not\equiv 1(m^2)$ or $\theta_D^4 \not\equiv 1 (D)$ then there is no residual term and we can use Theorem \ref{exi}. In this case, this gives $\Re(\mu_{l,j}) \leq \frac{\alpha}{4},$ where $\alpha$ is the best bound towards the Ramanujan conjecture for automorphic representations. Taking $\epsilon>\frac{\alpha}{8}$ the sum is bounded by $O(X^{\epsilon}).$ 

Let us label the terms \begin{equation} \label{eq:ress}\{\mbox{Remainder}\}:=\sum_{f_l \neq f_{00}}  \prod_{j=1}^2 \frac{ \widetilde{\Psi}(\mu_{l,j})}{g(-\mu_{l,j})^2}X^{\frac{\mu_{l,j}}{2}} \overline{c_{\frac{4DB^2}{16A},\sigma_a}(f_{l})}c_{\frac{B^2}{16A},\sigma_b}(f_{l})d_{\frac{4DB^2}{16A},\sigma_a}(\mu_l)\overline{d_{\frac{B^2}{16A},\sigma_b}(\mu_l)}+CSC_{\frac{4DB^2}{16A},\frac{B^2}{16A}}^{\sigma_a,\sigma_b}.\end{equation} Then by the previous paragraph $$\frac{\{\mbox{Remainder}\}}{\widetilde{\Psi}(\frac{-1}{4})^2X^{\frac{1}{4}}}=O(1).$$ 

The main term is
\begin{multline}\label{eq:wcon}   \frac{X^{\frac{1}{4}} \widetilde{\Psi}(\frac{1}{4})^2}{g(-\frac{1}{4})^2} (\frac{\pi}{2})^2\sum_{\substack{m \in \ri\\ m\equiv 1(4)}} \frac{\theta_D(m)\N(m)^{1/4}}{\phi(m)} \sum_{\substack{\chi \mod (m)\\ \chi^4 \equiv 1\mod (m)}} \tau(\chi) \chi(8A(\overline{-B^2}))  \times \\ \overline{c_{\frac{4DB^2}{16A},\sigma_a}(f_{00})}c_{\frac{B^2}{16A},\sigma_b}(f_{00})d_{\frac{4DB^2}{16A},\sigma_a}(\frac{1}{4})\overline{d_{\frac{B^2}{16A},\sigma_b}(\frac{1}{4})}+\{\mbox{Remainder}\}\end{multline} This is the main term of our Theorem. The equation \eqref{eq:wcon} converges as it is equal to the absolutely convergent sum \eqref{eq:hard}, and is clearly $O(X^{1/4}).$ However, it would be desirable to directly show that the $m$-sum is convergent. We make some remarks on this issue after completing the proof.

 We have \eqref{eq:wcon} equal to \begin{multline}\label{eq:wconv1}  \frac{X^{\frac{1}{4}}  \widetilde{\Psi}(\frac{1}{4})^2}{g(\frac{-1}{4})^2} (\frac{\pi}{2})^2  \sum_{\substack{m \in \ri\\ m\equiv 1(4)}} \frac{\theta_D( m)\N(m)^{1/4}}{\phi(m)} \sum_{\substack{\chi \mod (m)\\ \chi^4 \equiv 1\mod (m)}} \tau(\chi) \chi(8A(\overline{-B^2})) \times \\ \bigg[\overline{c_{\frac{4DB^2}{16A},\sigma_a}(f_{00})}c_{\frac{B^2}{16A},\sigma_b}(f_{00})d_{\frac{4DB^2}{16A},\sigma_a}(\frac{1}{4})\overline{d_{\frac{B^2}{16A},\sigma_b}(\frac{1}{4})}+\frac{g(-\frac{1}{4})^2\{\mbox{Remainder}\}}{X^{\frac{1}{4}}\widetilde{\Psi}(\frac{1}{4})}\bigg]=X^{\frac{1}{4}} \widetilde{\Psi}(\frac{1}{4})^2   K_{\frac{4DB^2}{16A},\frac{B^2}{16A}},\end{multline} Hence, \eqref{eq:harde} is equal to \eqref{eq:wconv1} which completes the theorem.

\end{proof}

\begin{remark}\label{rre}
\begin{itemize}

\item We recall from Theorem \ref{al} for $0< \epsilon<1$ we have \begin{multline} \{\mbox{Remainder}\}=\sum_{\substack{ \mu_l \mbox{ excep.}\\ \mu_l \neq \{\frac{1}{4},\frac{1}{4}\}}} \prod_{j=1}^2 \frac{ \widetilde{\Psi}(\mu_{l,j})}{g(-\mu_{l,j})}X^{\frac{\mu_{l,j}}{2}}  \overline{c_{\frac{DB^2}{16A},\sigma_a}(f_l)}c_{\frac{B^2}{16A},\sigma_b}(f_l)d_{\frac{DB^2}{16A},\sigma_a}(\mu_l)\overline{d_{\frac{B^2}{16A},\sigma_b}(\mu_l)} \\ +O(\log X+ X^{\epsilon-1}).\end{multline}
If we could show that we can interchange the $m$-sum along with the Residual Eisenstein series from the rest of the spectrum we could greatly simplify the main term.

The term above would simplify to a main term plus a error term depending on the Ramanujan conjecture, \begin{multline}\label{eq:hardf}\frac{X^{\frac{1}{4}} \widetilde{\Psi}(\frac{1}{4})^2}{g(\frac{-1}{4})^2}    \bigg[ \sum_{\substack{m \in \ri\\ m\equiv 1(4)}} \frac{\theta_D(m)\N(m)^{1/4}}{\phi(m)}  \times \\ \sum_{\substack{\chi \mod (m)\\ \chi^4 \equiv 1\mod (m)}} \chi(8A(\overline{-B^2})) \tau(\chi)  \overline{c_{\frac{4DB^2}{16A},\sigma_a}(f_{00})}c_{\frac{B^2}{16A},\sigma_b}(f_{00})d_{\frac{4DB^2}{16A},\sigma_a}(\frac{1}{4})\overline{d_{\frac{B^2}{16A},\sigma_b}(\frac{1}{4})}\bigg]+O(X^{\epsilon}),\end{multline} with $\epsilon > \frac{\alpha}{8}$ and $\alpha$ as defined before.

\item How do we understand the $m$-dependence of the metaplectic Fourier coefficients? In the automorphic setting, Deshouillers and Iwaniec (\cite{DeshI}, Theorem 2) prove that almost all Fourier coefficients for an automorphic form $f$ with level $q$ has the bound  $$c_{r,\sigma}(f) \ll \\\frac{1}{\nu(q)^{1/2}},$$ where $\nu(q)=q\prod_{p|q} (1+\frac{1}{p}).$ If that bound is true in this setting for our residual Eisenstein series of level $4Dm^2,$ the $m$-sum for the residual Eisenstein series-- applying trivial bounds everywhere else-- would absolutely converge. 
\end{itemize}

\end{remark}

\begin{proof}[Proof of Corollary \ref{corw}]
Just apply the first argument of the above Remark, then we have an asymptotic depending on the residual Eisenstein series plus the remainder term. The $\{Remainder\}$ then is bounded by $O(X^{\epsilon})$ using the Ramanujan conjecture.

\end{proof}

\section{Sums of quartic exponential sums}

In this section we prove our main Theorem \ref{ftt}. For clarity, we reiterate the assumptions made in the hypothesis of the theorem. 

Let $\Psi \in C^{\infty}_0(\R^{+})$ and $D \in \ri,D\equiv 1(4).$ Let $\theta$ be a Dirichlet character modulo $D$ with $A,B,F \in \Z[\omega_8].$ Assume that $B=4B'$ and $B'\parallel D^{\infty}$ with $\frac{B^2}{16A} \in \Z[\omega_8].$ Assume also that $A,B$ are squares.

 We investigate the asymptotic \begin{equation}\label{eq:aq}
\sum_{\substack{c \in \ri \\ c\equiv 1(4)\\(c,D)=1}} \frac{\prod_{j=1}^2\Psi(\frac{\sqrt{X}}{|c^{\eta_j}|^2})}{\N(c)}\theta(c)\sum_{x(c)} e(Tr(\frac{Ax^4+Bx^2+F}{c})).
\end{equation}

Using Proposition \ref{c400}, we can rewrite the sum \eqref{eq:aq} as \begin{multline}\label{eq:222}
\sum_{\substack{c \in \ri\\c\equiv 1(4)\\(c,D)=1}} \frac{\prod_{j=1}^2\Psi(\frac{\sqrt{X}}{|c^{\eta_j}|^2})}{\N(c)} \theta(c) e(\frac{F}{c})\sum_{x(c)} e(\frac{Ax^4+B^2x}{c})=\\ \sum_{\substack{c \in \ri\\c\equiv 1(4)\\(c,D)=1}} \frac{\prod_{j=1}^2\Psi(\frac{\sqrt{X}}{|c^{\eta_j}|^2})}{\N(c)} \theta(c) e(\frac{F}{c})(\frac{AB}{c})_2 e(\frac{-B^2\overline{8 A}}{c})S_4(\overline{4^2A}B^2,\overline{4^2A}B^2,c)+\\ \sum_{\substack{c \in \ri\\c\equiv 1(4)\\(c,D)=1}} \frac{\prod_{j=1}^2\Psi(\frac{\sqrt{X}}{|c^{\eta_j}|^2})}{\N(c)} \theta(c) e(\frac{F}{c}) \sum_{x(c)} e(\frac{Ax^2+Bx}{c})+\\ \{\mbox{``Cross terms" from Section \ref{sec:tro}}\}.
\end{multline} 


\subsection{First sum on RHS of \eqref{eq:222}}\label{sec:cs}

Note as $\Psi$ is compactly supported on the positive reals, $|c^{\eta_j}|^2 \gg X, j=\{1,2\}$ so for a fixed $G \in \ri,$ $$ Tr(\frac{G}{c})\ll X^{-1/2}.$$ So \begin{equation}\label{eq:egg} e( \frac{G}{c})=\sum_{j=0}^\infty \frac{(2\pi i Tr(\frac{G}{c}))^j}{j!}=1+O(X^{-1/2}).\end{equation}

As $\frac{B^2}{8A} \in \ri$ implies   $\frac{-B^2}{8A}\equiv -B^2\overline{8A}(c)$ for all $c \in \ri,$ we have by the above argument  \begin{equation} \label{eq:tamm}e(\frac{F}{c})e(\frac{-B^2\overline{8A}}{c})=e(\frac{F}{c})e(\frac{\frac{-B^2}{8A}}{c})=1+O(X^{-1/2}).\end{equation}

Therefore, we can write--using \eqref{eq:tamm} and the fact that $A,B$ are squares--the first sum of \eqref{eq:aq} as \begin{equation}\label{eq:aq01}\sum_{\substack{c \in \ri \\c\equiv 1(4)\\(c,D)=1}}  \frac{\prod_{j=1}^2\Psi(\frac{\sqrt{X}}{|c^{\eta_j}|^2})}{\N(c)} \theta(c)  S_{4}(\frac{B^2}{16A},\frac{B^2}{16A},c)+O(X^{\epsilon}),\end{equation} for any $\epsilon>0.$ \begin{remark}
 To analyze the next two sums from \eqref{eq:222}, the analogous procedure of using \eqref{eq:tamm} in removing $e(\frac{F}{c})$ is used and not mentioned again for completion of Theorem \ref{ftt}.\end{remark}

Using \eqref{eq:fsf} and Theorem \ref{calk}, for any $\epsilon>0$ we have \begin{multline}\label{eq:4q} \sum_{\substack{c \in \ri\\c\equiv 1(4)\\(c,D)=1}}  \frac{\prod_{j=1}^2\Psi(\frac{\sqrt{X}}{|c^{\eta_j}|^2})}{\N(c)} \theta(c)  S_{4}(\frac{B^2}{16A},\frac{B^2}{16A},c)=\\\sum_{\mu_l \mbox{ excep.}} \overline{c_{r,\sigma}(f_l)}c_{r',\sigma'}(f_l)d_{r,\sigma}(\mu_l)\overline{d_{r',\sigma'}(\mu_l)}\prod_{j=1}^2 \frac{ \widetilde{\Psi}(\mu_{l,j})}{g(-\mu_{l,j})}X^{\frac{\mu_{l,j}}{2}}
+O(X^{\epsilon}).\end{multline} Here the cusps $\sigma, \sigma'$ are defined in Theorem \ref{calk}.

With Theorem \ref{exi} we know there exists exceptional $\mu_l.$ Namely, there exists a residual Eisenstein series $f_{00}$ with $\mu_{00}=\frac{1}{4}$ if $\theta^4\equiv 1\mod(D),$ and further that there is no other exceptional  $\mu_l$ with $\mu_l > \frac{\alpha}{4}.$ Recall $\alpha$ is the best bound towards the Ramanujan conjecture for automorphic representations on $\F.$ 

So for any $\epsilon>\frac{\alpha}{8},$ \begin{multline}\label{eq:4q0}  \sum_{\substack{c \in \ri\\c\equiv 1(4)\\(c,D)=1}}  \frac{\prod_{j=1}^2\Psi(\frac{X}{|c^{\eta_j}|^2})}{\N(c)} \theta(c)  S_{4}(\frac{B^2}{16A},\frac{B^2}{16A},c)=\\ \frac{X^{\frac{1}{4}} \widetilde{\Psi}(\frac{1}{4})^2}{g(\frac{-1}{4})^2} (\frac{\pi}{2})^2 \overline{c_{\frac{B^2}{16A},\sigma}(f_{00})}c_{\frac{B^2}{16A},\sigma'}(f_{00})d_{\frac{B^2}{16A},\sigma}(\frac{1}{4})\overline{d_{\frac{B^2}{16A},\sigma'}(\frac{1}{4})} +O(X^{\epsilon}).\end{multline}

We recall $$g(\nu)=\prod_{j=1}^2 \frac{4(4\pi^2|(rr')_j|)^{\nu_j}\nu_j}{\Gamma(\nu_j+1)^2}$$ for a metaplectic form $f_l$ with eigenvalue parameter $\{\mu_{l,1},\mu_{l,2}\}=\{\nu_1,\nu_2\}$ and Fourier coefficient $c_{r,\sigma}(f_l).$ In this case, $$g(-\frac{1}{4})=[{(4\pi^2)^{1/4}\Gamma(3/4)^2 \N(\frac{B^2}{16A})^{1/8}}]^{-1}.$$


\subsection{Second sum on RHS of \eqref{eq:222}}

For the second sum we use Proposition \ref{q22} to get for any $M \geq 0,$ \begin{multline}\label{eq:qaf00}\sum_{\substack{c \in \ri\\ c\equiv 1(4)\\(c,D)=1}} \frac{\prod_{j=1}^2\Psi(\frac{\sqrt{X}}{|c^{\eta_j}|^2})}{\N(c)}\theta(c)\sum_{x(c)} e(Tr(\frac{Ax^2+Bx+F}{c}))=\\ \frac{X^{\frac{1}{2}}T_{A,D}(\theta) }{\phi(4)} \int_{\C^2} \Psi(\frac{1}{|t|^2})e(\frac{B^2}{4AX^{1/4}t})e(\frac{F}{X^{1/4}t})\frac{dt}{|t|}+O(X^{-M})\end{multline}  where \begin{equation}
 T_{A,D}(\theta)
 := \left\{ \begin{array}{ll} \N(A) & \text{if } \theta \equiv \mathbf{1}(D); \medskip \\
        0 & \text{if else}.\end{array} \right.  \end{equation}

\subsection{Third sum on RHS of \eqref{eq:222}}

For this term we use Theorem \ref{crrm} from Section \ref{sec:tro}: 
$$\{\mbox{``Cross terms" from Section \ref{sec:tro}}\}=$$

\begin{multline}\label{eq:harde7} \sum_{\substack{n,m \in \ri\\ (n,m)=1\\ nm\equiv 1(4)}} \frac{\theta_D(nm)\prod_{j=1}^2\Psi(\frac{\sqrt{X}}{|(nm)^{\eta_j}|^2})}{\N(nm)} \left[\sum_{x(m)}e(\frac{\overline{n}(Ax^2+Bx)}{m})\right]\left[e(\frac{-B^2\overline{8 Am}}{n})S_4(-B^2\overline{m4^2A},n)\right]=\\ X^{\frac{1}{4}} \widetilde{\Psi}(\frac{1}{4})   K_{\frac{DB^2}{16A},\frac{B^2}{16A}}.\end{multline}

\subsection{Conclusion}


After the analysis of these three terms, we conclude for any $\epsilon>\frac{\alpha}{8},$  \begin{multline}\label{eq:two}
\sum_{\substack{c \in \ri\\ c\equiv 1(4)\\(c,D)=1}} \prod_{j=1}^2 \frac{\Psi(\frac{\sqrt{X}}{|(c)^{\eta_j}|^2})}{\N(c)}\theta(c)\sum_{x(c)} e(Tr(\frac{Ax^4+Bx^2+F}{c}))=\\ \frac{X^{\frac{1}{2}}T_{A,D}(\theta) }{\phi(4)} \int_{\C^2} \Psi(\frac{1}{|t|^2})e(\frac{B^2}{4AX^{1/4}t})e(\frac{F}{X^{1/4}t})\frac{dt}{|t|}+\\   \frac{X^{\frac{1}{4}} \widetilde{\Psi}(\frac{1}{4})^2}{g(\frac{-1}{4})^2}(\frac{\pi}{2})^2 \overline{c_{\frac{B^2}{16A},\sigma}(f_{00})}c_{\frac{B^2}{16A},\sigma'}(f_{00})d_{\frac{B^2}{16A},\sigma}(\frac{1}{4})\overline{d_{\frac{B^2}{16A},\sigma'}(\frac{1}{4})}+\\ X^{\frac{1}{4}} \widetilde{\Psi}(\frac{1}{4})^2   K_{\frac{DB^2}{16A},\frac{B^2}{16A}}+O(X^{\epsilon}). \end{multline} 

Recall again that $\sigma, \sigma'$ are cusps defined in Section \ref{sec:fig}. 

Let $\hat{\psi}(s)$ denote the Mellin transform over $\R,$ $$\hat{\psi}(s)=\int_0^\infty \psi(y)y^{s-1}dy.$$

To make the result more symmetric in $\Psi,$ we can write using the same argument as in \eqref{eq:egg}, $$\int_{\C^2} \Psi(\frac{1}{|t|^2})e(\frac{B^2}{4AX^{1/4}t})e(\frac{F}{X^{1/4}t})\frac{dt}{|t|}=\int_{\C^2} \psi(\frac{1}{|t|^2})\frac{dt}{|t|}+O(X^{-1/2}).$$ Then using a change of variables to polar coordinates $$\int_{\C^2} \Psi(\frac{1}{|t|^2})\frac{dt}{|t|}=\frac{(2\pi)^3}{(2\sqrt{2})^2} [ \int_{0}^\infty \Psi(\frac{1}{y^2})dy]^2=\frac{\pi^3}{2} \hat{\Psi}(-\frac{1}{2})^2.$$

The error from using the Taylor expansion above is of size $$X^{1/2} \cdot O(X^{-1/2})=O(1)$$ and we contain it in the error $O(X^{\epsilon}).$



Recalling that $\widetilde{\Psi}(s)=\int_0^\infty \Psi(\frac{1}{t}) t^{s} \frac{dt}{t}=\hat{\Psi}(-s),$  we get \eqref{eq:two} equals \begin{multline}\label{eq:three}
\sum_{\substack{c \in \ri\\ c\equiv 1(4)\\(c,D)=1}} \frac{\prod_{j=1}^2\Psi(\frac{\sqrt{X}}{|c^{\eta_j}|^2})}{\N(c)}\theta(c)\sum_{x(c)} e(Tr(\frac{Ax^4+Bx^2+F}{c}))=\\ X^{\frac{1}{2}} \frac{\pi^3}{2\phi(4)} \hat{\Psi}(-\frac{1}{2})^2  T_{A,D}(\theta) +\\  
\frac{X^{\frac{1}{4}} \hat{\Psi}(-\frac{1}{4})^2}{g(-\frac{1}{4})^2} (\frac{\pi}{2})^2 \overline{c_{\frac{B^2}{16A},\sigma}(f_{00})}c_{\frac{B^2}{16A},\sigma'}(f_{00})d_{\frac{B^2}{16A},\sigma}(\frac{1}{4})\overline{d_{\frac{B^2}{16A},\sigma'}(\frac{1}{4})} +\\
X^{\frac{1}{4}} \hat{\Psi}(-\frac{1}{4})^2   K_{\frac{DB^2}{16A},\frac{B^2}{16A}}+O(X^{\epsilon})\end{multline} for any $\epsilon > \frac{\alpha}{8}.$ This completes Theorem \ref{ftt}.

\section{Appendix: Bruggeman-Miatello Kuznetsov trace formula over $\F=\Q[\omega_8]$ with multiple cusps}\label{sec:app}

\begin{definition}
For $\gamma>0$ and $a=(a_1,a_2)$ with $a_j>3$ define $\mathcal{H}_a(\gamma)$ as the set of functions $h=h_1 \times h_2$ such that $h_j$ is an even holomorphic function on $|\Re \nu| \leq 2\gamma$ satisfying $$h_j(\nu) \ll e^{\frac{-|\Im \nu|}{2}}(1+|\Im \nu|)^{-a_j}.$$

\end{definition}

\begin{definition}
For $r \in \ri \setminus \{0\},$ and $h \in \mathcal{H}_a(\gamma),$ define the function $K_r h$ on $G$ by $$K_rh(g):=\sqrt{\N(r)} \prod_{j=1}^2 \frac{i}{2} \int_{\Re(\nu_j)=0} h_j(\nu_j)  W_{\sigma}^{\nu_j,r}(g)    \sin \pi \nu_j d\nu_j$$ with $W_{\sigma} ^{\nu,r}(ng_{\sigma}a[y]k):=\chi_r(n)\sqrt{\N(y)}\prod_{j=1}^2 K_{\nu_j}(4\pi y_j\sqrt{|r_j|}).$ 
\end{definition}

Fix non-equivalent essential cusps $\sigma, \sigma'.$ We have the Poincare series $$P_{\sigma,\psi}^r(g)=\sum_{\Gamma_{N^{\sigma}}\setminus \Gamma} (\theta(\gamma )\kappa(\gamma ))^{-1}\chi_r(n(\gamma g)) \psi(a(\gamma g)),$$ with $\psi(a(g)):=K_rh(a(g)).$ Note it transforms by $P_{\sigma,\psi}^r(\gamma g)=\theta(\gamma)\kappa(\gamma) P^r_{\sigma,\psi}(g).$ The absolute convergence of the series, and in particular the following integrals of this variant of the Bruggeman-Miatello trace formula follow immediately from the analysis of [\cite{BM}, Section 10], and we will not mention such issues anymore.

Take a nontrivial square integrable automorphic form $f_l(g).$ Then following directly the procedure for taking the inner product of a Poincare series with an automorphic form in [\cite{BMP1},5.3] we have $$\langle P_{\sigma,\psi}^r, f_l \rangle=\frac{\pi^2}{2} \overline{c_{r,\sigma}(f_l)} d_{r,\sigma}(\mu_l) \sqrt{\N(r)}h(\mu_l).$$

We recall from [\cite{BM},\cite{BMP1}]  that $c_{r,\sigma}(f_l)$ is $r$-th Fourier coefficient of the automorphic form $f_l$ at the cusp $\sigma$ and $$d_{r,\sigma}(\nu)=\frac{1}{vol(\Gamma_{N^{\sigma}}\backslash N^{\sigma})} \frac{2^{1-\nu}(4\pi\sqrt{N(r)})^{-\nu}}{\Gamma(\nu+1)}.$$

\begin{remark}
There seems to be a discrepancy between the factor $d_{r,\sigma}(\nu)$ in \cite{BM} and \cite{BMP1} for totally real fields. But checking with \cite{IK} and \cite{L}, it should be as we defined it.
\end{remark}

\subsection{Spectral computation of the scalar product}

Using the standard spectral decomposition of an inner product of automorphic forms $\langle f,g \rangle$ with the above defined Poincare series $P_{\sigma,\psi}^r$ as well as $P_{\sigma',\psi'}^{s}$ with $\psi':=K_rh', h' \in \mathcal{H}_a(\gamma)$ we have $$\langle P_{\sigma,\psi}^r, P_{\sigma',\psi'}^{s} \rangle=( \frac{\pi^2}{2} )^2\sqrt{\N(rs)}\int_{Y}k(\nu)d\zeta_{r,s,\sigma,\sigma'}(\nu).$$

Here for an even test function $\eta$ \begin{multline}\int_Y \eta(\nu)d\zeta_{r,s,\sigma,\sigma'}(\nu)=\sum_{l \geq 1} \overline{c_{r,\sigma}(f_l)}c_{s,\sigma'}(f_l)d_{r,\sigma}(\mu_l)\overline{d_{s,\sigma'}(\mu_l)}\eta(\mu_l)+\\ \sum_{\beta \in \mathcal{P}} \sum_{\mu} \int_{\infty}^\infty \overline{D_{r,\sigma}(P^{\beta},iy\rho,i\mu)}D_{s,\sigma'}(P^{\beta},iy\rho,i\mu)d_{r,\sigma}(\mu_l)\overline{d_{s,\sigma'}(\mu_l)}\eta(iy\rho+i\mu)dy.\end{multline}

\subsection{Geometric computation of the scalar product}
Following [\cite{BMP1},(58)] we can expand the inner product as 

$$\langle P_{\sigma,\psi}^r, P_{\sigma',\psi'}^{s} \rangle=\sum_{\gamma \in \Gamma_{N^{\sigma'}}\setminus \Gamma} I(\gamma),$$ with  $$I(\gamma)=(\theta(\gamma)\kappa(\gamma))^{-1}\int_A a^{-2\rho}\psi(g_{\sigma} a) \int_{\Gamma_{N^{\sigma}}\backslash N^{\sigma}} \chi_r(n) \psi'(\gamma n g_{\sigma}a)dn da.$$

We write this as $I_1+I_2,$ 
$$I_1:=\sum_{\Gamma_{N^{\sigma'}} \setminus (\Gamma \cap g_{\sigma'}Pg_{\sigma}^{-1})} I(\gamma),$$ and $$I_2:=\sum_{\gamma \in \Gamma_{N_{\sigma'}} \setminus {}^{\sigma'}{\Gamma}^{\sigma}} I(\gamma).$$

We remind that ${}^{\sigma'}{\Gamma}^{\sigma}=\Gamma \cap g_{\sigma'}Cg_{\sigma}^{-1}$ with $C=(P \cap G_{\Q})s_0(N \cap G_{\Q})$ and
$s_0=\left( \begin{array}{cc}
0 & -1 \\ 1 & 0 \\ \end{array} \right).$

Define the measure $$\int_Y f(\nu)d\eta(\nu):=\prod_{j=1}^2 \frac{i}{2}\int_{\Re(\nu_j)=0}f_j(\nu)\sin \pi\nu_j d\nu_j.$$
For $I_1$ we write just as in \cite{BM}, $$I_1=(\frac{\pi^2}{2})^4 \sqrt{\N(rs)}\Delta_{r,s}^{\sigma,\sigma'}(k)=(\frac{\pi^2}{2})^4
\sqrt{\N(rs)} vol(\Gamma_{N^{\sigma}}\backslash N^{\sigma})\alpha(\kappa,r,\kappa',s) \int_Y k(\nu)d
\eta(\nu)$$ with $k(x)=h(x)\overline{h'}(x)$ and $\alpha(\kappa,r,\kappa',r')=
\sum_{\gamma} (\theta(\gamma)\kappa(\gamma))^{-1}\chi_{r'}(n_{\sigma}(\gamma g_{\sigma})^{-1})$ 
with the sum over $\gamma \in \Gamma_{N^{\sigma'}} \setminus (\Gamma \cap g_{\sigma'}
Pg_{\sigma}^{-1})$ for which $\chi_r(n)=\chi_{s}(\gamma n \gamma^{-1})$ for all $n \in N^{\sigma}.$ 
This is similar to the Definition $2.6.1$ in \cite{BMP1}, and equivalent to the number of $a_{\gamma}=g_{\sigma} a_{\sigma}(\gamma g_{\sigma}) g_{\sigma}^{-1}$ such that $a_{\gamma} \cdot r= s.$

Using the exact same reasoning as \cite{BMP1},  $$I_2=\sum_{{}^{\sigma'}\mathcal{L}^{\sigma}} \sum_{\delta \in \Gamma_{N^{\sigma}}} I(\gamma \delta).$$ Let $c \in {}^{\sigma'}\mathcal{C}^{\sigma}, \gamma \in {}^{\sigma'}\mathcal{L}^{\sigma}(c)$ following their Definition $2.3.1$ and Proposition $2.3.2.$ 
Write $\xi=g_{\sigma'}^{-1}\gamma g_{\sigma}=\left( \begin{array}{cc}
a(\gamma) & b(\gamma) \\ c(\gamma) & d(\gamma) \\ \end{array} \right)=n[\xi]m[\xi]a_{\xi}s_0n'[\xi]\in SL_2(\F),$ with \begin{equation}\label{eq:exp}n[\xi]m[\xi]a_{\xi}s_0n'[\xi]=\left( \begin{array}{cc}
1 & \frac{a(\gamma)}{c(\gamma)} \\ 0 & 1 \\ \end{array} \right)\left( \begin{array}{cc}
\frac{1}{c(\gamma)} & 0 \\ 0 & c(\gamma) \\ \end{array} \right)s_0\left( \begin{array}{cc}
1 & \frac{d(\gamma)}{c(\gamma)} \\ 0 & 1 \\ \end{array} \right).\end{equation}

Writing $c=c(\gamma)$ for economy, we have \begin{multline}\sum_{\delta \in \Gamma_{N^{\sigma}}} I(\gamma \delta)=(\theta(\gamma)\kappa(\gamma))^{-1} \int_A a^{-2\rho} K_rh(g_{\sigma}a)\int_{N^{\sigma}} \chi_r(n) \overline{K_{s}h'(\gamma ng_{\sigma}a)}dnda=\\ \pi \sqrt{\N(s/r)} (\theta(\gamma)\kappa(\gamma))^{-1} \overline{\chi_{s}(n'(\xi))\chi_r(n(\xi))}\int_A a^{-2\rho} K_rh(g_{\sigma}a) K_{s} h'_{1/c^2}(g_{\sigma}a)da\end{multline} with $h_t(\nu):=\sqrt{N(t)} h(\nu) \mathbf{B}(rst^2,\nu)$ and $$\mathbf{B}(t,\nu):=\prod_{j=1}^2 \frac{I_{\nu_j}(4\pi\sqrt{t_j})I_{\nu_j}(4\pi\overline{\sqrt{t_j}})-I_{-\nu_j}(4\pi\sqrt{t_j})I_{-\nu_j}(4\pi\overline{\sqrt{t_j}})}{\sin \pi \nu_j}.$$ Again, this notation is analogous to \cite{BM} as we are working over a field with only complex embeddings. 
Following $(83)$ of \cite{BMP1}, $$\int_A a^{-2\rho}K_rh(g_{\sigma}a)K_{s}h'_{1/c^2}(g_{\sigma}a)da=2\pi^3 \N(r) \int_Y h(\nu)h'_{1/c^2}(\nu)d\eta(\nu).$$ Opening the definition of $h_{1/c^2}$ and writing $k=h\overline{h'},$ \begin{multline}\label{eq:ii}\sum_{\delta \in \Gamma_{N^{\sigma}}} I(\gamma \delta)=(\frac{\pi^3}{2})^2  \sqrt{\N(sr)} (\theta(\gamma)\kappa(\gamma))^{-1} \overline{\chi_{s}(n'(\xi))\chi_r(n(\xi))} |\N(c)|^{-1} \int_Y k(\nu) \mathbf{B}(\frac{rs}{c^2},\nu)d\eta(\nu)=\\ (\frac{\pi^3}{2})^2 \sqrt{\N(sr)} (\theta(\gamma)\kappa(\gamma))^{-1} \overline{\chi_s(n'(\xi))\chi_r(n(\xi))} |\N(c)|^{-1}  Bk(\frac{rs}{c^2}).\end{multline}
The last equality uses Definition 5.7 from \cite{BM}.

Now recalling our dependence of $a,b,c,d$ on $\gamma,$ the expression \eqref{eq:exp} gives $$\overline{\chi_{s}(n'(\xi))\chi_r(n(\xi))}=\overline{e(Tr(\frac{sd(\gamma)}{c(\gamma)}+\frac{ra(\gamma)}{c(\gamma)}))}.$$ 

Let $${}^{\sigma'}\Gamma^{\sigma}(c)=\{ \gamma \in {}^{\sigma'}\Gamma^{\sigma} : g_{\sigma'}^{-1}\gamma g_{\sigma}=\left( \begin{array}{cc}
\cdot & \cdot \\ c & \cdot \\ \end{array} \right)\},$$ then Proposition $2.3.2$ from \cite{BMP1} has an analogous argument for a number field with only complex embeddings which can be expressed for a test function $f$ as

 \begin{equation}\label{eq:rew}\sum_{\gamma \in {}^{\sigma'}\mathcal{L}^{\sigma}}f(\gamma) =\sum_{c \in {}^{\sigma'}\mathcal{C}^{\sigma}} \sum_{\gamma \in 
\Gamma_{N^{\sigma'}}\setminus {}^{\sigma'}\Gamma(c)^{\sigma} / \Gamma_{N^{\sigma}}}f(\gamma).\end{equation}

Label ${}^{\sigma'}\mathcal{L}(c)^{\sigma}:=\Gamma_{N^{\sigma'}}\setminus {}^{\sigma'}\Gamma(c)^{\sigma} / \Gamma_{N^{\sigma}}.$ Using \eqref{eq:ii} and \eqref{eq:rew}, \begin{equation}
I_2=\sum_{{}^{\sigma'}\mathcal{L}^{\sigma}} \sum_{\delta \in \Gamma_{N^{\sigma}}} I(\gamma \delta)=  (\frac{\pi^3}{2})^2 \sum_{c \in {}^{\sigma'}\mathcal{C}^{\sigma}} S^{\sigma,\sigma'}_{\kappa\theta}(r,s,c)\sqrt{\N(sr)}  |\N(c)|^{-1}  Bk(\frac{rs}{c^2}) \end{equation}

where we have defined for a multiplicative character $\Delta$ on $SL_2(\ri),$ $$S^{\sigma,\sigma'}_{\Delta}(r,s,c):=\sum_{\left( \begin{array}{cc}
a &b \\ c & d \\ \end{array} \right) \in 
g_{\sigma'}^{-1}{}^{\sigma'}\mathcal{L}(c)^{\sigma}g_{\sigma}} \Delta( g_{\sigma'} \left( \begin{array}{cc}
a &b \\ c & d \\ \end{array}  \right)g_{\sigma}^{-1})^{-1} e(-Tr(\frac{sd}{c}+\frac{ra}{c})).$$

So we have expressed the inner product of two Poincare series in two different ways giving 
\begin{multline}\label{eq:bmtr}
(\frac{\pi^2}{2})^4\sqrt{\N(rs)}\bigg[\sum_{l \geq 1} \overline{c_{r,\sigma}(f_l)}c_{s,\sigma'}(f_l)d_{r,\sigma}(\mu_l)\overline{d_{s,\sigma'}(\mu_l)}k(\mu_l)+\\ \sum_{j=1}^m \sum_{\mu} \int_{\infty}^\infty \overline{D_{r,\sigma}(P^j,iy\rho,i\mu)}D_{s,\sigma'}(P^j,iy\rho,i\mu)d_{r,\sigma}(\mu_l)\overline{d_{s,\sigma'}(\mu_l)}k(iy\rho+i\mu)dy\bigg]=\\ (\frac{\pi^2}{2})^4
\sqrt{\N(rs)} vol(\Gamma_{N^{\sigma}}\backslash N^{\sigma})\alpha(\sigma,r,\sigma',s) \times \\ \int_{Y} k(\nu)\sin \pi \nu d\nu+(\frac{\pi^3}{2})^2 \sqrt{\N(sr)} \sum_{c \in {}^{\sigma'}\mathcal{C}^{\sigma}} S^{\sigma,\sigma'}_{\kappa\theta}(r,s,c) |\N(c)|^{-1}  Bk(\frac{rs}{c^2}).
\end{multline}

Dividing by $(\frac{\pi^2}{2})^4\sqrt{\N(rs)}$ we get \begin{multline}\label{eq:bmtr0} \sum_{l \geq 1} \overline{c_{r,\sigma}(f_l)}c_{s,\sigma'}(f_l)d_{r,\sigma}(\mu_l)\overline{d_{s,\sigma'}(\mu_l)}k(\mu_l)+\\ \sum_{j=1}^m \sum_{\mu} \int_{\infty}^\infty \overline{D_{r,\sigma}(P^j,iy\rho,i\mu)}D_{s,\sigma'}(P^j,iy\rho,i\mu)d_{r,\sigma}(\mu_l)\overline{d_{s,\sigma'}(\mu_l)}k(iy\rho+i\mu)dy\bigg]=\\ 
vol(\Gamma_{N^{\sigma}}\backslash N^{\sigma})\alpha(\kappa,r,\kappa',s) \int_{Y} k(\nu)\sin \pi \nu d\nu+(\frac{2}{\pi})^2  \sum_{c \in {}^{\sigma'}\mathcal{C}^{\sigma}} S^{\sigma,\sigma'}_{\kappa\theta}(r,s,c) |\N(c)|^{-1}  Bk(\frac{rs}{c^2}).
\end{multline}

This trace formula holds with the choices of $h,h' \in \mathcal{H}_a(\gamma)$ and $k=h\overline{h'}.$ Like  \cite{BM}, we give such $k$ a definition: 

\begin{definition}
For $\gamma>0$ and $a=(a_1,a_2)$ with $a_j>6$ define $\mathcal{K}_a(\gamma)$ as the set of functions $k=k_1\times k_2$ with $k_j$ an even holomorphic function on $|\Re \nu| \leq 2\gamma$ satisfying $$k_j(\nu) \ll e^{-|\Im \nu|}(1+|\Im \nu|)^{-a_j}.$$

\end{definition}

Compare this to Theorem $2.7.1$ of \cite{BMP1} and Theorem $6.1$ of \cite{BM}.

\section{Appendix 2: Exponential sum identities}

We look at the case of Theorem \ref{c4} for a higher power. 

\begin{prop}\label{cnn}
Let $p,A,B \in \Z[\omega_{2n}],$ with $p\equiv 1 (n),(AB,p)=1.$ Denote the character of order $2$ as $\eta.$ Then, we have the following identity $$\sum_{x(p)} e(\frac{Ax^{2n}+Bx^n}{p})=\sum_{\substack{\rho(n)\\ \xi^2=\rho}} \frac{\rho(B^2)\tau(\eta)\eta(-1)}{\N(p)} \sum_{a,b(p)} \xi(a)\eta(a) \xi(b) e(\frac{a+b+4ab\overline{B^2}A}{p}).$$

\end{prop}

\begin{proof}
\begin{equation}
\sum_{A(p)}\overline{\chi(A)} \sum_{x(p)}  \rho(x) e(\frac{Ax^2+Bx}{p}).
\end{equation}

By a change of variables $A \to A \overline{x^2}, x \to \overline{B}x$ this equals $$\chi(\overline{B^2})\rho(B^2)\tau(\chi^2 \rho)\tau(\overline{\chi}).$$

So by Fourier inversion \begin{equation}\label{eq:fi}\sum_{x(p)}  \rho(x) e(\frac{Ax^2+Bx}{p})= \rho(B^2) \frac{1}{\phi(p)} \sum_{\chi(p)} \chi(\overline{B^2}) \tau(\chi^2 \rho)\tau(\overline{\chi})\chi(A).\end{equation}

The Hasse-Davenport relation gives the identity for $\chi(p),$ \begin{equation}\label{eq:hd}
\tau(\chi^n)=\frac{-\chi(n^n) \prod_{l(n)} \tau(\chi \gamma^{l})}{\prod_{l(n)} \tau(\gamma^l)},
\end{equation} where $\gamma$ is the $n$-th residue character.

Suppose $\rho=\xi^2,$ one can consider $\xi$ a $2n$-th residue character.

Writing the term $\tau(\chi^2 \rho)=\tau((\chi \xi)^2)=\tau((\chi \xi)^2),$ and using the Hasse-Davenport relation we can write 

$$\tau((\chi \xi)^2)=\frac{\chi(4) \tau(\chi \xi \eta)\tau(\chi \xi)\tau(\eta)\eta(-1)}{\N(p)}.$$ This equality can be reached by using the equality $\overline{\tau(\eta)}=\eta(-1)\tau(\eta).$

So \eqref{eq:fi} equals $$ \frac{\rho(B^2)\tau(\eta)\eta(-1)}{\N(p)\phi(p)} \sum_{\chi(p)} \chi(\overline{B^2}) \chi(A) \left[\chi(4) \tau(\chi \xi \eta)\tau(\chi \xi)\tau(\overline{\chi})\right].$$

Opening all the Gauss sums and rearranging sums this equals $$\frac{\eta(B)\tau(\eta)\eta(-1)}{\N(p)\phi(p)} \sum_{a(p)} \xi(a)\eta(a)e(\frac{a}{p}) \sum_{b(p)} \xi(b)e(\frac{b}{p}) \sum_{c(p)} e(\frac{c}{p})\sum_{\chi(p)} \chi(4ab\overline{cB^2}A).$$

Using orthogonality of characters this equals 
$$\frac{\rho(B^2)\tau(\eta)\eta(-1)}{\N(p)} \sum_{\substack{a,b,c(p)\\ 4abA\equiv B^2c(p)}} \xi(a)\eta(a)e(\frac{a}{p})  \xi(b)e(\frac{b}{p})  e(\frac{c}{p}).$$

We can rewrite this as $$\frac{\rho(B^2)\tau(\eta)\eta(-1)}{\N(p)} \sum_{a,b(p)} \xi(a)\eta(a) \xi(b) e(\frac{a+b+4ab\overline{B^2}A}{p}).$$

But unlike Theorem \ref{c4}, this cannot be reduced to a rank $2$ Kloosterman sum. This sum is a close relation to a hypergeometric sum found in \cite{K}. So the power $n=2$ is a special case that the exponential sums in question are related to rank $2$ Kloosterman sums. 


\end{proof}

If one had to guess what exponential sums are associated to rank $2$ Kloosterman sums, one could look at the prime power modulus, say $p^m, m>1$ for $p$ in some number field. From Proposition \ref{klo}, it is necessary that the polynomial of the exponential sum looks like $f(x)=Ax^n+Bx^{n-2}+C$ for $A,B,C$ integers in some number field. Indeed, Proposition \ref{klo} tells us that the derivative needs to satisfy some relation $F(A,B,C)x^2\equiv G(A,B,C) (p^j),$ for $1 \leq j \leq m,$ with $F(A,B,C), G(A,B,C)$ some functions of $A,B,C.$ One can imitate the same method done for Proposition \ref{cnn} above for such an exponential sum $$\sum_{x(p)} e(\frac{Ax^n+Bx^{n-2}+C}{p}),$$ and realize that is not related, at least via an analogous Hasse-Davenport argument, to a Kloosterman sum. Even for $n=4,$ it is not clear what to do, until one applies characters of order 2 to the sum and then one arrives at Theorem \ref{c4}. All of this argument is not a proof that no other exponential sums of higher degree polynomials are related to  rank $2$ Kloosterman sums, but is good evidence the degree 4 sums we study are likely the highest power related to rank $2$ Kloosterman sums and therefore for metaplectic forms of any cover of $GL_2.$ 

It is clear these sums can be related to hyper-Kloosterman sums or higher rank Kloosterman or Hypergeometric sheafs of \cite{K}, but an important question is to link a sum of such sums or sheafs to something spectral like automorphic forms via the trace formula. This seems to be a difficult question as it is still not clear whether rank $3$ Kloosterman sheafs are related to $GL_3$ automorphic forms or some subset (automorphic representations of $U(3)$?) of them.

\end{document}